\documentclass{amsart}
\title[%
   Jorge-Meeks type maximal surfaces]{%
   Analytic extension of Jorge-Meeks type maximal surfaces 
   in Lorentz-Minkowski 3-space
}
\date{September 19, 2015}
\usepackage[dvipdfmx]{graphicx}
\usepackage{enumerate}
\usepackage{verbatim}
\theoremstyle{plain}
 
 \newtheorem{theorem}{Theorem}[section]
 \newtheorem*{theorem*}{Theorem}

 \newtheorem*{lemma*}{Lemma}
 \newtheorem{proposition}[theorem]{Proposition}
 
 \newtheorem*{fact*}{Fact}

 \newtheorem{lemma}[theorem]{Lemma}
 \newtheorem{corollary}[theorem]{Corollary}
 \theoremstyle{remark}

  \newtheorem{definition}[theorem]{Definition}
 
 \newtheorem*{remark*}{Remark}
 \newtheorem*{problem*}{Problem}

\numberwithin{equation}{section}
\numberwithin{figure}{section}

\newcommand{\Z}{\boldsymbol{Z}}

\newcommand{\R}{\boldsymbol{R}}
\newcommand{\C}{\boldsymbol{C}}

\newcommand{\imag}{\mathrm{i}}

\renewcommand{\Re}{\operatorname{Re}}
\renewcommand{\Im}{\operatorname{Im}}
\newcommand{\trans}[1]{{\vphantom{#1}}^t#1}
\renewcommand{\phi}{\varphi}
\renewcommand{\epsilon}{\varepsilon}

\renewcommand{\mid}{\, ; \,}
\author{S.~Fujimori}
\address[Shoichi Fujimori]{%
   Department of Mathematics, Okayama University,
   Tsushima-naka, Okayama 700-8530, Japan}
\email{fujimori@math.okayama-u.ac.jp}
\author{Y. Kawakami}
\address[Yu Kawakami]{%
   Graduate School of Natural Science and Technology,
   Kanazawa University,
   Kanazawa, 920-1192, Japan,
}
\email{y-kwkami@se.kanazawa-u.ac.jp}

\author{M. Kokubu}
\address[Masatoshi Kokubu]{%
   Department of Mathematics, School of Engineering, 
   Tokyo Denki University, 
   Tokyo 120-8551, Japan}
\email{kokubu@cck.dendai.ac.jp}

\author{W.~Rossman}
\address[Wayne Rossman]{%
   Department of Mathematics, Faculty of Science,
   Kobe University,
   Rokko, Kobe 657-8501, Japan
}
\email{wayne@math.kobe-u.ac.jp}

\author{M.~Umehara}
\address[Masaaki Umehara]{%
   Department of Mathematical and Computing Sciences,
   Tokyo Institute of Technology
   2-12-1-W8-34, O-okayama, Meguro-ku,
   Tokyo 152-8552, Japan.
}
\email{umehara@is.titech.ac.jp}

\author{K.~Yamada}
\address[Kotaro Yamada]{%
   Department of Mathematics\\
   Tokyo Institute of Technology\\
   O-okayama, Meguro, Tokyo 152-8551\\
   Japan
}
\email{kotaro@math.titech.ac.jp}


\begin{document}
\begin{abstract}
 The Jorge-Meeks $n$-noid ($n\ge 2$)
 is a complete minimal surface
 of genus zero with $n$ catenoidal ends
 in the Euclidean 3-space $\R^3$, which has
 $(2\pi/n)$-rotation symmetry with respect to its axis. 
 In this paper,
 we show that the corresponding maximal surface $f_n$
 in Lorentz-Minkowski 3-space $\R^3_1$ has 
 an analytic extension $\tilde f_n$
 as a properly embedded 
 zero mean curvature surface.
 The extension changes type into a time-like (minimal) surface.
\end{abstract}
\maketitle
\section*{Introduction}
\begingroup
\renewcommand{\theequation}{\arabic{equation}}
\renewcommand{\thefigure}{\arabic{figure}}
A number of 
zero mean curvature surfaces 
of mixed type in Lorentz-Minkowski three-space 
$(\R^3_1;t,x,y)$ were found in \cite{K1}, 
\cite{G}, \cite{ST}, \cite{Kl}, 
\cite{FRUYY2}, \cite{FKKRSUYY1} and \cite{FKKRSUYY2}.
One of the main tools for the construction 
of such surfaces is based on the fact that fold 
singularities of space-like maximal surfaces have  
real analytic extensions to time-like minimal surfaces 
(cf.\ \cite{G}, \cite{Kl}, \cite{KKSY} and \cite{FKKRSUYY2}). 
Some of the analytic extensions of such examples 
have neither singularities nor self-intersections.
A typical such example is a space-like helicoid,
which analytically extends to a time-like surface,
and the entire surface coincides with the original
helicoid as a minimal surface in $\R^3$.
Also, the Scherk type surface
\begin{equation}\label{eq:S}
 t(x,y):=\log\frac{\cosh y}{\cosh x}
\end{equation}
gives an entire graph which changes type from
a space-like maximal surface 
to a time-like zero mean curvature surface, as
pointed out by Kobayashi \cite{K1}.
Recently, it was shown in \cite{FRUYY2} 
that the space-like maximal analogues in 
$\R^3_1$ of the Schwarz D surfaces in $\R^3$
have analytic extensions as triply periodic
embedded zero mean curvature surfaces.
\begin{figure}[htb]
 \centering
 \includegraphics[height=4cm]{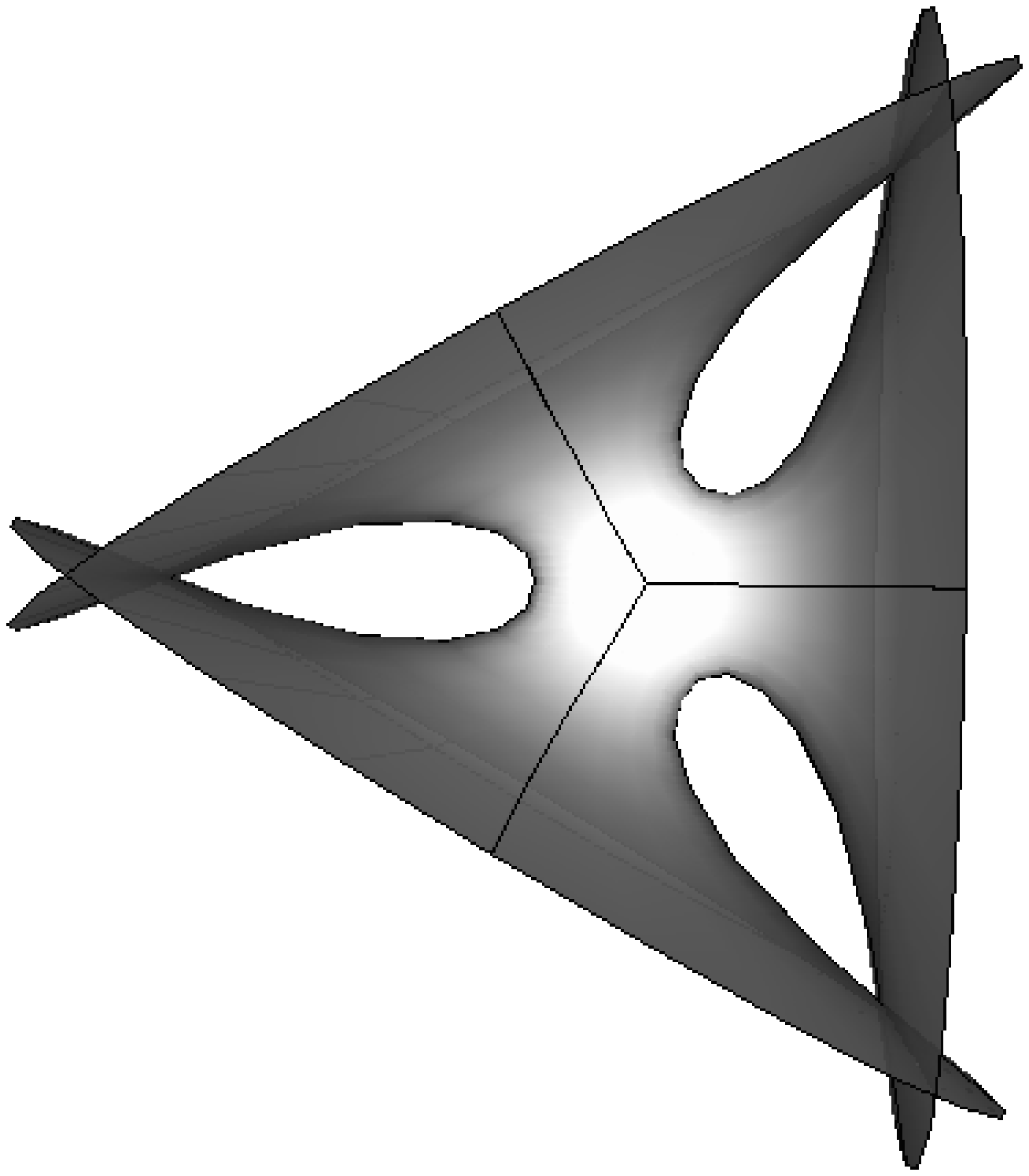} \hskip 1.7cm 
 \includegraphics[height=4.6cm]{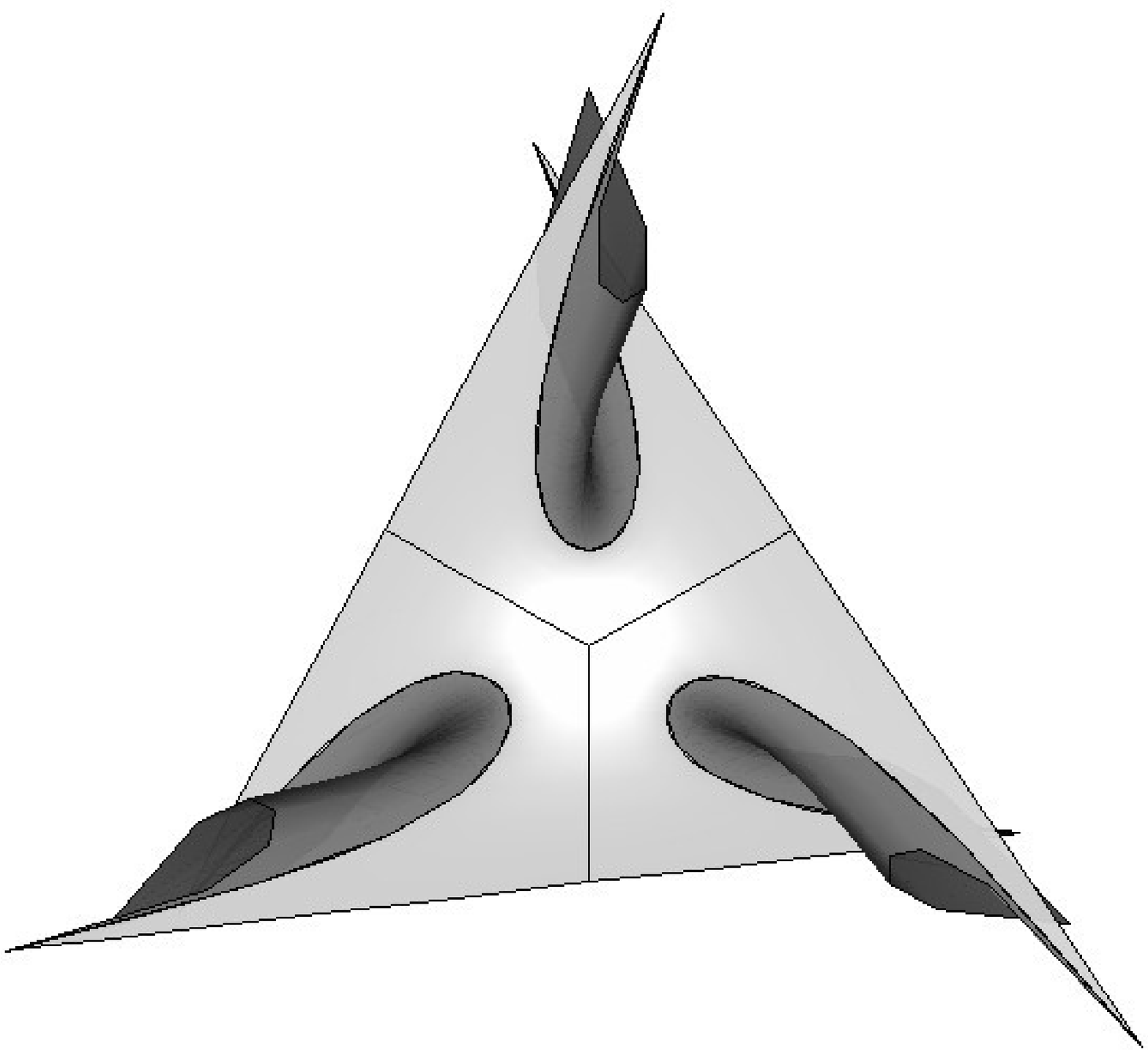}  
\caption{%
 The Jorge-Meeks trinoid in $\R^3$
 and the analytic extension of $f_3$
 (in the figure on the right-hand side, 
the time-like parts are
indicated by black shading).}
 \label{fig:00}
\end{figure}
These examples caused the authors
to be interested in space-like maximal analogues
$f_n$ ($n=2,3,\dots$) of  
Jorge-Meeks minimal surfaces with $n$ catenoidal ends.
These surfaces have fold singularities, and have
analytic extensions to time-like surfaces.
We show in this paper that 
the analytic extension of $f_n$ is a proper embedding. 
\endgroup
\section{Preliminaries}
\label{sec:1}

We denote by $(\R^3_1;t,x,y)$
the Lorentz-Minkowski $3$-space of signature $(-++)$
and denote 
the Riemann sphere by $S^2:=\C\cup\{\infty\}$.

\begin{definition}\label{def:W}
 A pair $(g,\omega)$ consisting of a meromorphic function and a 
 meromorphic $1$-form  defined on the Riemann sphere
 is called a \emph{Weierstrass data on $S^2$}
 if the metric
 \begin{equation*}
     ds^2_E:=(1+|g|^2)^2 |\omega|^2
 \end{equation*}
 has no zeros on $S^2$.
 A point where $ds^2_E$ diverges
 is called an \emph{end} of $ds_E^2$.
\end{definition}
We now fix a Weierstrass data $(g,\omega)$ on $S^2$ and
let $\{p_1,\dots,p_n\}$ be the set of ends of $ds^2_E$. 
Then the real part of the map
\begin{equation*}
    F:=
        \int_{z_0}^z\bigl(-2g, 1+g^2,\imag (1-g^2)\bigr)\omega
	\qquad\left(\imag=\sqrt{-1}\right)
\end{equation*}
is a map
\begin{equation*}
 f_L=\Re (F),
\end{equation*}
into $\R^3_1$
which is defined on the  universal cover of 
$S^2\setminus\{p_1,\dots,p_n\}$.
We call $f_L$ the
\emph{maximal surface associated to  $(g,\omega)$},
and $F$ the \emph{holomorphic lift} of $f_L$.
If $f_L$ is single-valued on $S^2\setminus\{p_1,\dots,p_n\}$,
then we say that $f_L$ 
satisfies the \emph{period condition}.
The first fundamental form of $f_L$ is given by
\begin{equation*}
    ds^2=(1-|g|^2)^2 |\omega|^2.
\end{equation*}
In particular,
the singular set of $f_L$ consists of the points where
$|g|=1$.  
In this situation, 
we set $F=(X_0,X_1,X_2)$.
As pointed out in \cite{UY},
the real part $f_E:=\Re(F_E)$
of the holomorphic map
$F_E:=(X_1,X_2,\imag X_0)$
gives a conformal minimal immersion into
the Euclidean $3$-space $\R^3$
defined on the universal cover
of $S^2\setminus\{p_1,\dots,p_n\}$
such that the first fundamental
form of $f_E$ coincides with
$\pi^*ds^2_E$, where $\pi$ is the covering projection.
In particular, $F_E$ (and also $F$)
is an immersion.
So the map $f_L$ is a \emph{maxface} in the sense 
of \cite{UY} (see also \cite{FSUY} 
and \cite{FKKRSUYY2}, in particular, a convenient 
definition of maxface which is equivalent to 
the original one is given in 
\cite[Definition 2.7]{FKKRSUYY2}).
The minimal immersion $f_E$ is called
the \emph{companion} of $f_L$.

We are interested in the maximal surface $f_n$ associated to 
\begin{equation*}
   g_n=z^{n-1},\qquad
   \omega_n=\frac{\imag dz}{(z^n-1)^2}
   \qquad (n=2,3,4,\dots).
\end{equation*}
As pointed out in \cite[Example 5.7]{UY},
the companion of $f_n$
is congruent to 
the well-known complete minimal surface
with catenoidal ends, 
called a \emph{Jorge-Meeks surface} 
(cf.\ \cite{JM}).
In particular, the associated metric $ds^2_E$
is complete on 
\begin{equation*}
 S^2 \setminus \{1,\zeta,\dots, \zeta^{n-1}\}
\text{ where } \zeta:=e^{2\pi \imag/n}.
\end{equation*}
This means that $(g_n,\omega_n)$ is a Weierstrass data on $S^2$. 
It can be checked that $f_n$ is single-valued on
$S^2 \setminus \{1,\zeta,\dots, \zeta^{n-1}\}$, and
the original Jorge-Meeks surface is as well.
So $f_n$ is a maxface, and we call $\{f_n\}_{n=2,3,\dots}$ 
the \emph{Jorge-Meeks type maximal surfaces}. 
The singular set of $f_n$ is the set $|z|=1$,
which consists of generic fold singularities
in the sense of \cite{FKKRSUYY2}, that is,
the image of the singular set consists
of a union of non-degenerate null curves
in $\R^3_1$.

\begin{figure}[thbp] 
\centering
 \includegraphics[width=.33\linewidth]{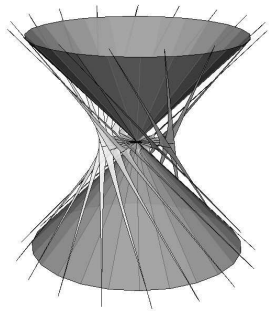} 
\caption{Jorge-Meeks type maximal surface $\tilde f_{17}$ 
(one-sheeted hyperboloid is also shown).
}
\label{fg:hyp-cyl0}
\end{figure}

We now observe that
$f_2$ has a canonical analytic extension 
embedded in $\R^3_1$ (see Figure \ref{fig:01}):
By definition,
\[
  f_2=\Re
      \left(
         \frac{\imag}{z^2-1},
         -\frac{\imag z}{z^2-1},
	 \frac12\log \frac{1-z}{1+z}
      \right).
\]
If we set $f_2=(x_0,x_1,x_2)$ 
and $z=r e^{\imag \theta}$, 
then
\begin{align*}
 x_0&=\frac{r^2 \sin 2 \theta}{r^4-2 r^2 \cos 2 \theta+1},
    \quad
 x_1=-\frac{r \left(r^2+1\right) \sin \theta}
      {r^4-2 r^2 \cos 2 \theta+1},\\
 x_2&=\frac{1}{4} \log \left(
     \frac{r^2-2 r \cos\theta+1}{r^2+2 r \cos \theta+1}\right).
\end{align*}
In particular, it holds that
\[
   \frac{x_0}{x_1}=-\frac{2 r \cos \theta}{r^2+1}
       =\tanh 2x_2.
\]
Thus, the image of $f_2$ is a subset of the
graph  $t=x\tanh 2y$ (Figure~\ref{fig:01}, left), 
and it changes type on the set
\[
   S:=\left\{\left(\pm \frac{\cosh 2 y}2,y\right)
          \,;\, y\in \R\right\},
\]
in the $xy$-plane,
and the connected domain with boundary $S$ consists of 
the image of the orthogonal projection of $f_2$  
into the $xy$-plane (cf.\ Figure \ref{fig:01}, right).
This means that the image of $f_2$ has
an analytic extension that coincides with
a zero mean curvature entire graph,
like as in the case of the Scherk type surface
\eqref{eq:S} in the introduction.

\begin{figure}[htb]
 \centering
  \includegraphics[height=4cm]{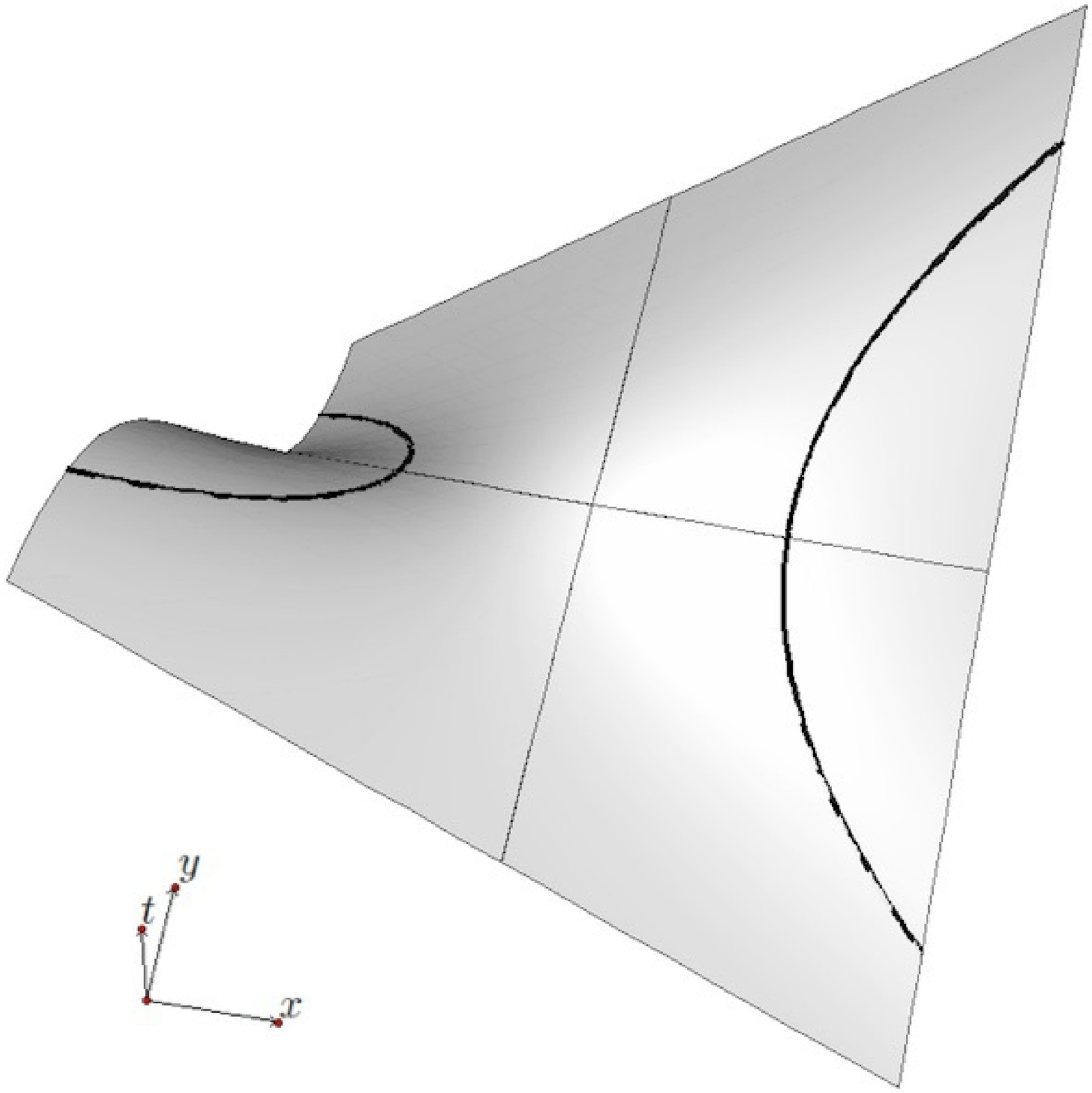} \hskip 1.7cm 
  \includegraphics[height=3.5cm]{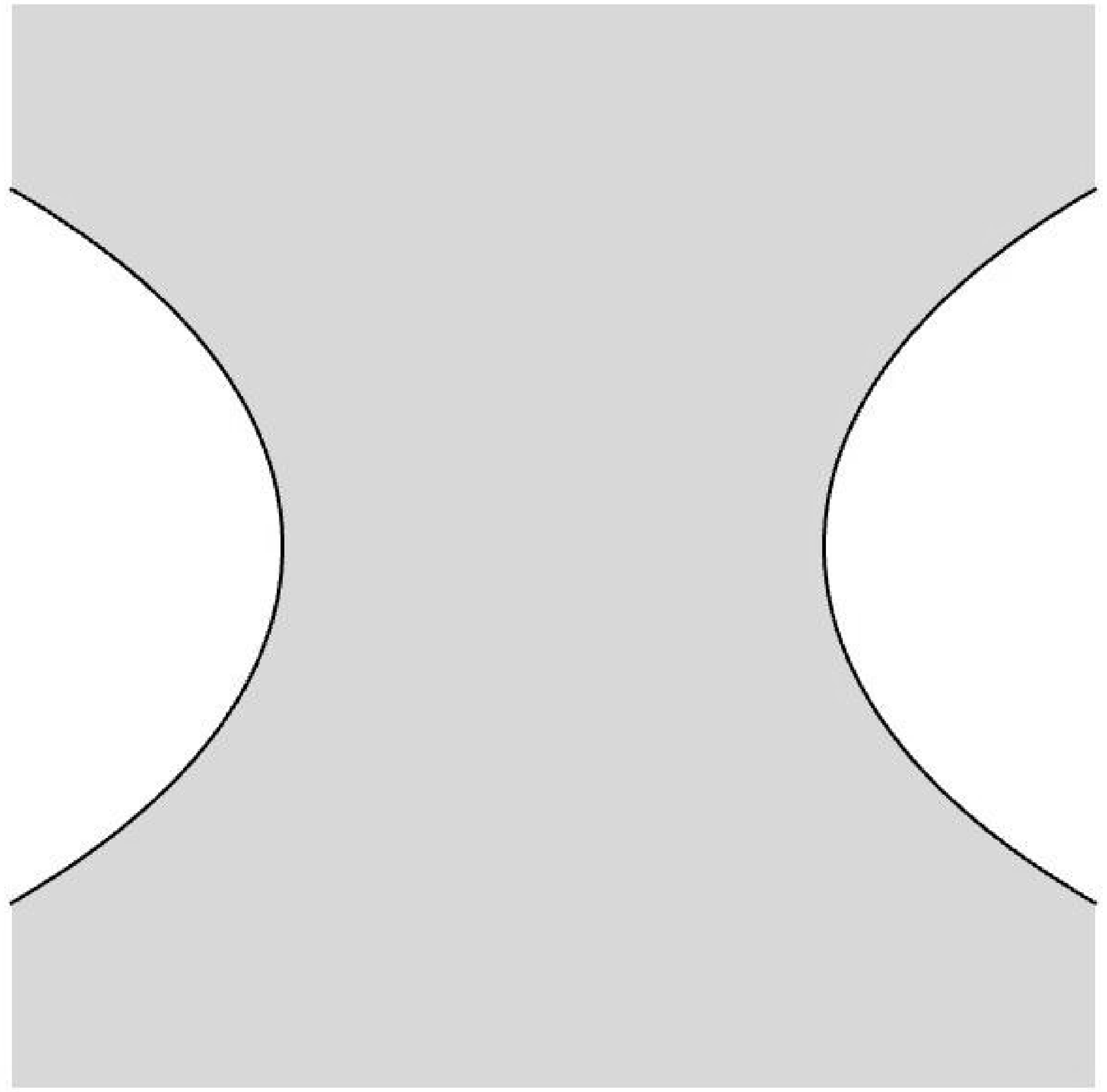}  
 \caption{%
  The analytic extension of $f_2$ 
  and its orthogonal projection of 
the space-like part.}
  \label{fig:01} 
\end{figure}

In \cite{K1}, the maximal surface $f_2$ is called a 
\textit{helicoid of the 2nd kind} and it was already pointed out that 
the function $t= x \tanh 2y$ is an entire solution to 
the maximal surface equation for graphs. 
The conjugate surface of $f_2$
induces a singly periodic
maxface, called the \emph{hyperbolic catenoid}
(see \cite{FKKRSUYY2} for details).

\section{Analytic extension of $f_n$}
\label{sec:2}
In this section, we will show that 
each $f_n$ 
has a canonical analytic extension
for $n\ge 3$
as well.
By definition, 
the Jorge-Meeks type maximal
surface $f_n=(x_0,x_1,x_2)$ and its holomorphic lift
$F=(X_0,X_1,X_2)$ are  given by
\begin{equation*}
  f_n =\Re (F) , \qquad F = \int_0^z \alpha,
\end{equation*}
where we set
\begin{equation}\label{eq:alpha}
 \begin{aligned}
  \alpha&=\alpha(z) = a(z)\,dz=\bigl(a_0(z),a_1(z),a_2(z)\bigr)\,dz\\
        &:=
          \left(
	    -\frac{2 \imag z^{n-1}}{\left(z^n-1\right)^2},
	     \frac{\imag \left(1+z^{2 n-2}\right)}{\left(z^n-1\right)^2},
            -\frac{1-z^{2 n-2}}{\left(z^n-1\right)^2}
	   \right)\,dz.
\end{aligned}
\end{equation}
Using these expressions, we show the following:
\begin{proposition}\label{prop:symmetry}
Regarding $f_n$ as a column vector-valued function,
the image of $f_n$ has the following two properties{\rm:}
 \[
    f_n(\bar z)=S
    f_n(z),\qquad
    f_n(\zeta z)
     =Rf_n(z),
 \]
 where
 \begin{equation}\label{eq:SR}
    S:=
      \left(
       \begin{array}{rrr}
       -1 & 0 & 0 \\
        0 & -1 & 0 \\
	0 & 0 & 1 
       \end{array}
      \right),\quad
      R:=
      \left(
       \begin{array}{ccc}
	1 & 0 & 0 \\
	0 & \hphantom{-}\cos \frac{2\pi}{n} & \sin \frac{2\pi}{n} \\[6pt]
	0 & -\sin \frac{2\pi}{n} & \cos \frac{2\pi}{n} 
       \end{array}
      \right)
 \end{equation}
and 
\begin{equation*}
 \zeta:=e^{2\pi\imag/n}.
\end{equation*}
\end{proposition}
\begin{proof}
 Since
 \[
   R=\frac{1}{2} \left(
     \begin{array}{ccc}
      2 & 0 & 0 \\
      0 & \zeta +\zeta^{-1} & -\imag \left(\zeta -\zeta^{-1}\right) \\
      0 & \imag \left(\zeta -\zeta^{-1}\right) & \zeta+\zeta^{-1} \\
     \end{array}
   \right),
 \]
 the $1$-form $\alpha$ in \eqref{eq:alpha} satisfies
 \[
   \overline{\alpha(\bar z)}=
     S\alpha(z),
    \qquad
    \alpha(\zeta z)=
       R\alpha(z),
 \]
 where $\alpha$ is considered as a column vector-valued 
 $1$-form.
 Since $F(0)=0$, we have $\overline{F(\bar z)}=S F(z)$
 and $F(\zeta z)=R F(z)$.
 In particular, we have  the relations
 $f_n(\bar z)=S f_n(z)$ and
 $f_n(\zeta z)=R f_n(z)$.
\end{proof}

\begin{lemma}\label{lem:lift-explicit}
 Up to a suitable translation in $\C^3$ by a vector in
$\imag \R^3$,
 the holomorphic lift
 $F=(X_0,X_1,X_2)$ of the Jorge-Meeks type maximal 
 surface $f_n$ has the following expression{\rm:}
 \begin{align}\label{eq:X0}
  X_0 &= \frac{2\imag}{n(z^n-1)},\\
  \label{eq:X1}
  X_1 &= 
         -\imag
            \left(
              \frac{z(z^{n-2}+1)}{n(z^n-1)}
                 +\frac{n-1}{n^2}
                   \sum_{j=1}^{n-1}
                  (\zeta^j-\zeta^{-j})\log (z-\zeta^j)\right),\\
  \label{eq:X2}
  X_2 &= 
           -\frac{z(z^{n-2}-1)}{n(z^n-1)}
             +\frac{n-1}{n^2}\sum_{j=0}^{n-1}(\zeta^j+\zeta^{-j})
              \log(z-\zeta^j).
 \end{align}
\end{lemma}
\begin{proof}
 The first identity \eqref{eq:X0} is obvious.
To prove the second identity \eqref{eq:X1}, 
we will show that differentiation of the 
right-hand side of \eqref{eq:X1} is equal to
 $a_1(z)$.
Denoting the right-hand side of \eqref{eq:X1} by $\hat X_1$,
we have that
 \[
\frac{d \hat X_1}{dz} - a_1(z)=
    \frac{d \hat X_1}{dz}
    -\frac{\imag(1+z^{2 n-2})}{\left(z^n-1\right)^2}\\
     =
- \imag\frac{(n-1) \phi(z)}{n^2 \left(z^n-1\right)},
 \]
 where we set 
 \[
    \phi(z):=n \left(z^{n-2}-1\right)
     + 
     \sum _{j=1}^{n-1} \frac{(\zeta^j-\zeta^{-j})(z^n-1)}{z-\zeta^j}. 
 \]
For $z=\zeta^k$ ($k=0,1,2, \dots , n-1$),
\begin{equation}\label{eq:phi-zeros}
 \begin{aligned}
 \phi(\zeta^k) &= n \left(\zeta^{-2k}-1\right)
     + 
     \sum _{j=1}^{n-1} (\zeta^j-\zeta^{-j})
\left. \frac{z^n-1}{z-\zeta^j} \right|_{z=\zeta^k}\\
&=
 n \left(\zeta^{-2k}-1\right)
      +\left(\zeta^{k}-\zeta^{-k}\right) n\zeta^{-k} \\
&=0.
\end{aligned}
\end{equation}
Here we have used the following identity:
 \[
    \left.
     \frac{z^n-1}{z-\zeta^j}\right|_{z=\zeta^k}=
\begin{cases}
0 & \text{ if $j\ne k$}  \\
     \left.\dfrac{d}{dz}\left(z^n-1\right)\right|_{z=\zeta^k}
      =n \zeta^{-k} & \text{ if $j=k$. }
\end{cases}
 \]
\eqref{eq:phi-zeros} means that the number of zeros for $\phi(z)$ is at least $n$. 
However, $\phi(z)$ is a polynomial in $z$ of degree at most $n-1$.
So we conclude that $\phi$ vanishes identically, and hence 
$d \hat X_1/dz - a_1(z)=0$. 
 
 Similarly, 
denoting the right-hand side of \eqref{eq:X2} by $\hat X_2$, we have
 \[
\frac{d \hat X_2}{dz} - a_2(z)=
   \frac{d \hat X_2}{dz}
      +\frac{1-z^{2 n-2}}{\left(z^n-1\right)^2}=
     -\frac{(n-1) \psi(z)}{n^2 \left(z^n-1\right)},
 \]
 where $\psi(z)$ is a polynomial of degree at most $n-1$ given by 
 \[
    \psi(z):=n \left(z^{n-2}+1\right)-
       \sum _{j=0}^{n-1} 
      \frac{(\zeta^{-j}+\zeta^j)(z^n-1)}{z-\zeta^j}.
 \]
 It can be easily checked that
 $\psi(\zeta^k)=0$ for each $k=0,1,2,\dots,n-1$. 
These prove that $d \hat X_2/dz - a_2(z)=0$, and thus \eqref{eq:X2} is verified. 
\end{proof}

Using Lemma \ref{lem:lift-explicit}, we obtain an integration-free formula of 
$f_n$ as follows.  

\begin{proposition}\label{prop:x012}
 The Jorge-Meeks type maximal surface $f_n=(x_0,x_1,x_2)$
 has the following expressions{\rm:}
 \begin{align}
  \label{eq:x0}
    x_0 &=\frac{2r^n\sin n\theta}{n(r^{2n}-2 r^n \cos n\theta+1)},\\
  \label{eq:x1}
    x_1 &=-\frac{(r^{2n-1}+r)\sin\theta+(r^{n+1}+r^{n-1})\sin  (n-1)\theta}
               {n(r^{2n}-2r^n\cos n\theta+1)} \\
     \nonumber
     &\phantom{aaaaaaaaaaaaa}+
       \frac{n-1}{n^2}
       \sum_{j=1}^{n-1}
        \log\left(
            r^2 - 2 r \cos\left(\theta-\frac{2\pi j}{n}\right)+1
            \right)
        \sin\frac{2\pi j}{n},\\
  \label{eq:x2}
    x_2&= 
        \frac{-(r^{2n-1}+r)\cos\theta+(r^{n+1}+r^{n-1})\cos(n-1)\theta}
               {n(r^{2n}-2r^n\cos n\theta+1)} \\
     \nonumber
       &\phantom{aaaaaaaaaaaa}+
       \frac{n-1}{n^2}
        \sum_{j=0}^{n-1}
           \log\left(
               r^2 - 2 r \cos\left(\theta-\frac{2\pi j}{n}\right)+1
           \right)
             \cos\frac{2\pi j}{n},
 \end{align}
 where $z=re^{\imag\theta}$.
\end{proposition}

\begin{proof}
 Since
 \begin{align*}
  x_0&=\Re X_0 = -\frac{2\Im(\bar z^n-1)}{n(z^n-1)(\bar z^n-1)}\\
     &=-\frac{2\Im(r^n e^{-\imag n\theta})}{n(r^{2n}-2r^n\cos n\theta+1)}
    =\frac{2r^n\sin n\theta}{n(r^{2n}-2r^n \cos n\theta +1)},
 \end{align*}
 the first identity \eqref{eq:x0}
 is obtained.
 Similarly, one can easily check that
 \begin{align*}
   \Re\left(-\frac{\imag z \left(z^{n-2}+1\right)}{%
          n \left(z^n-1\right)}\right)
   &=-\frac{\left(r^{2 n-1}+r\right)\sin \theta 
                +\left(r^{n+1}+r^{n-1}\right)\sin(n-1)\theta}{
            n \left(r^{2 n}-
            2 r^n \cos n\theta+1\right)},\\
   \Re\left(
     -\frac{z \left(z^{n-2}-1\right)}{n \left(z^n-1\right)}
      \right)&=
        \frac{-(r^{2n-1}+r)\cos\theta
               +(r^{n+1}+r^{n-1})\cos(n-1)\theta}{
              n(r^{2n}-2r^n\cos n\theta+1)}.
 \end{align*}
 On the other hand,
 \begin{align*}
  \Im&
      \sum_{j=1}^{n-1}(\zeta^j-\zeta^{-j})\log(z-\zeta^j)
  =
  \sum_{j=1}^{n-1}
   2\sin\frac{2\pi j}{n}\log |z-\zeta^j|\\
  & =
   \sum_{j=1}^{n-1}
   \sin\frac{2\pi j}{n}\log |z-\zeta^j|^2
  =
   \sum_{j=1}^{n-1}
   \sin\frac{2\pi j}{n}\log\bigl((z-\zeta^j)(\bar z-\zeta^{-j})\bigr)\\
  &=
  \sum_{j=1}^{n-1}
    \log
    \left(
      r^2 - 2r\cos\left(\theta-\frac{2\pi j}{n}\right)+1
    \right)   \sin\frac{2\pi j}{n},
 \end{align*}
 which proves \eqref{eq:x1}.
Similarly, we have \eqref{eq:x2}.
\end{proof}

The following assertion is an immediate consequence of
Proposition \ref{prop:x012}.

\begin{corollary}\label{rem:why-fold}
$f_n$ satisfies the identity
\begin{equation}
 f_n(1/r, \theta) = f_n(r, \theta)
\qquad (r>0,\,\, 0\le \theta<2\pi).
\end{equation}
\end{corollary}

Since $f(r,\theta)$ is 
invariant under the symmetry
$r \mapsto 1/r$,
the singular set $\{|z|=1\}$ of $f$
coincides with
the fixed point set under the symmetry.
We remark that
the set $\{|z|=1\}$ consists of 
non-degenerate fold singularities as in
\cite{FKKRSUYY2}.
So, it is natural to introduce a 
new variable $u$ by
 \begin{equation}\label{eq:u-r}
    u := \frac{r+r^{-1}}{2},
\end{equation}
which is invariant under the symmetry
$r \mapsto 1/r$.
We set
$$
\bar D_1^*:=\{z\in \C\,;\, 0<|z|\le 1\}.
$$
By Corollary \ref{rem:why-fold},
$f(\bar D_1^*\setminus\{
1,\zeta,\dots,\zeta^{n-1}
\})$ coincides
with the whole image of $f$.
To obtain the analytic extension of $f$,
we define an analytic map
$$
\iota:
\bar D_1^*\ni z=r e^{\imag \theta} \mapsto 
(\frac{r+r^{-1}}2,\theta)\in 
\R\times \R/2\pi \Z.
$$
The image of the map $\iota$ is
given by
$$
\hat \Omega_n:=\{(u,\theta)\in 
\R\times \R/2\pi \Z\,;\,
u \ge 1
\}.
$$
The map $\iota$ is bijective,
whose inverse is given by 
$$
\iota^{-1}:
\R\times \R/2\pi \Z\ni (u,\theta)
\mapsto 
(u-\sqrt{u^2-1},\theta)\in \bar D_1^*.
$$
Using the Chebyshev polynomials, 
the formulas \eqref{eq:x0}--\eqref{eq:x2} can be 
rewritten  in terms of $(u,\theta)$ 
as follows
(see the appendix for the definition and basic properties 
of the Chebyshev polynomials). 

\begin{corollary}\label{lem:x012}
By setting $\tilde f_n=f_n\circ \iota^{-1}$
and $\tilde f_n=(\tilde x_0,\tilde x_1,\tilde x_2)$,
it holds that
 \begin{align}
 \label{eq:u0}
   \tilde x_0 &
    =\frac{\sin n\theta}{n(T_n(u)-\cos n\theta)},\\
 \label{eq:u1}
   \tilde x_1 &
    =-\frac{T_{n-1}(u)\sin \theta+u \sin (n-1)\theta}{%
           n(T_n(u)-\cos n\theta)}\\
            \nonumber
    &\phantom{aaaaaaaaaaaaa}+
      \frac{n-1}{n^2}
       \sum_{j=1}^{n-1}
              \log\left(
                 u- \cos\left(\theta-\frac{2\pi j}{n}\right)
              \right)
         \sin\frac{2\pi j}{n},\\
 \label{eq:u2}
   \tilde x_2&
    = \frac{-T_{n-1}(u)\cos \theta +u \cos (n-1)\theta}
               {n(T_n(u)-\cos n\theta)} \\
           \nonumber
   &\phantom{aaaaaaaaaaaa}+
      \frac{n-1}{n^2}
       \sum_{j=0}^{n-1}
              \log\left(
                 u- \cos\left(\theta-\frac{2\pi j}{n}\right)
              \right)
         \cos\frac{2\pi j}{n},
 \end{align}
where $T_n(u)$, $T_{n-1}(u)$ denote the first Chebyshev polynomials in 
the variable $u$ 
of degree $n$, $n-1$, respectively.
\end{corollary}
\begin{proof}
 Since
 \[
    x_0 =\frac{2r^n\sin n\theta}{n(r^{2n}-2 r^n \cos n\theta+1)}
         = \frac{\sin n \theta}{%
             n\left(\frac12(r^n+r^{-n})-\cos n \theta\right)},
 \]
 \eqref{eq:T2} in the appendix yields \eqref{eq:u0}.
 Similarly, the first terms of \eqref{eq:x1} and \eqref{eq:x2}
 are the same as the first terms of \eqref{eq:u1} and \eqref{eq:u2},
 respectively.
 On the other hand,
 \begin{align*}
  \sum_{j=1}^{n-1}&
    \log\left(r^2-2 r \cos\left(\theta-\frac{2\pi j}{n}\right) +1 \right)
    \sin\frac{2\pi j}{n}\\
   &=
  \sum_{j=1}^{n-1}
    \log
       \left(2 r \left(\frac{r+r^{-1}}{2}-\cos\left(\theta-\frac{2\pi
                   j}{n}\right)\right)\right)
    \sin\frac{2\pi j}{n}\\
   &=
     \sum_{j=1}^{n-1}
         \log \left(u-\cos\left(\theta-\frac{2\pi j}{n}\right)\right)
           \sin\frac{2\pi j}{n}
        +
          \log 2 r \sum_{j=1}^{n-1}
           \sin\frac{2\pi j}{n}.
 \end{align*}
 Then we have \eqref{eq:u1} because
 \[
    \sum_{j=1}^{n-1} \sin\frac{2\pi j}{n}
    =\Im
          \sum_{j=0}^{n} \zeta^j
      =0.
 \]
 Similarly, we have \eqref{eq:u2}.
\end{proof}

If we consider $\tilde f_n$ instead of $f_n$,
the origin $z=0$ in the source space of 
$f_n$ does not lie in that of $\tilde f_n$.
To indicate what the origin in the old
comlex coordinate $z$ becomes in the
new real coordinates $(u,\theta)$,
we attach a new point $p_{\infty}$
to $\hat \Omega_n$
as the `point at infinity',
and extend the map $\iota$ so that
$$
\iota(0)=p_\infty.
$$
Hence we have a one-to-one correspondence 
between $\{|z| \le 1 \}$ 
and $\hat \Omega_n\cup\{p_\infty\}$.
In particular, 
$\hat \Omega_n\cup\{p_\infty\}$ can be considered as  an analytic $2$-manifold.
We prove the following:

\begin{proposition}\label{prop:extf}
The map $\tilde f_n:\hat\Omega_n\cup\{p_\infty\}\to \R^3_1$ 
can be analytically extended to the domain
\begin{equation}\label{eq:omega}
  \Omega_n :=
   \left\{
    (u,\theta)\in \R\times \R/2\pi\Z\,;\,
      u> \max_{j=0,\dots,n-1}
             \left[\cos\left(\theta-\frac{2\pi j}n\right)\right]
     \right\}\cup \{p_\infty\}, 
\end{equation}%
\end{proposition}

\begin{proof}
In fact, \eqref{eq:u0}--\eqref{eq:u2} are meaningful if  
\begin{equation}\label{eq:meaningful}
 T_n(u)-\cos n \theta >0 \text{ and }
u > \cos \left( \theta -\frac{2 \pi j}{n} \right) \ (j=0,1, \dots , n-1).  
\end{equation}
Moreover, $T_n(u)-\cos n\theta$ is factorized as 
(cf. Lemma \ref{lemma:Phi} in the appendix)
\begin{equation}\label{eq:tnu-cos}
 T_n(u)-\cos n\theta  
       =2^{n-1}\prod_{j=0}^{n-1}
         \left(u-\cos\left(\theta-\frac{2\pi j}{n}\right)\right).
\end{equation}%
So the condition \eqref{eq:meaningful} reduces to 
\begin{equation*}
u > \cos \left( \theta -\frac{2 \pi j}{n} \right) \ (j=0,1, \dots , n-1).  
\end{equation*}
Thus, the components $\tilde x_j(u, \theta)$ of $\tilde f_n$
given in
\eqref{eq:u0}, \eqref{eq:u1} and \eqref{eq:u2}
can be extended to $  \Omega_n$.
\end{proof}

By Proposition \ref{prop:extf},
we may assume that the map $\tilde f_n$
is defined in $\Omega_n$. From now on,
we call this newly obtained analytic map
$$
\tilde f_n \colon \Omega_n \to \R^3_1
$$ 
the \emph{analytic extension} of $f_n$.

\bigskip

\begin{figure}[htb]
 \centering
 \includegraphics[height=5.6cm]{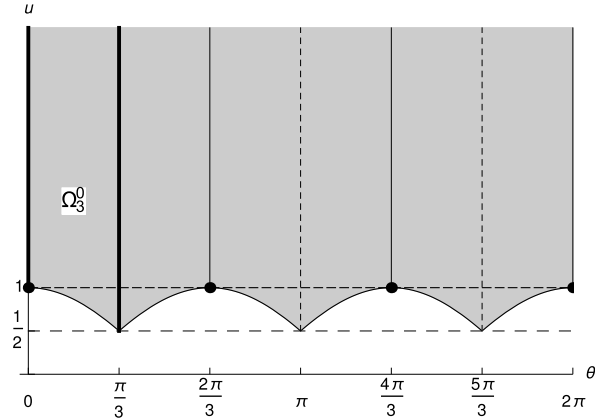} 
 \caption{%
  The domain $\Omega_n$ and the fundamental domain
  $\Omega_n^0$ for $n=3$, where
  the region of $u > 1$ is the space-like part
  and the regions of $u < 1$ are the 
time-like parts.}
 \label{fig:02}
\end{figure}

Proposition \ref{prop:symmetry} and the real analyticity of
$\tilde f_n$ imply that
\begin{equation}
 \label{eq:1}
    \tilde f_n(u,-\theta)
      =
    S\tilde f_n(u,\theta),\quad
   \tilde f_n\left(u,\theta+\frac{2\pi}{n}\right)
      =
     R\tilde f_n(u,\theta),
\end{equation}
where $S$ and $R$ are the matrices 
given in \eqref{eq:SR}, and 
$\tilde f_n$ is considered as a 
column vector-valued function.
Let $G$ be the finite isometry group 
of $\R^3_1$ generated by $S$ and $R$.
The subset 
\begin{equation}\label{eq:Omega0}
\Omega_n^0=\left\{ (u, \theta) \mid u>\cos \theta, \ 
0 \le \theta \le \frac{\pi}{n} \right\}
\end{equation}
of $\Omega_n$ is called the {\it fundamental domain} of $\tilde f_n$.
The equation \eqref{eq:u0} yields the following proposition.

\begin{proposition}\label{prop:G}
 The whole image of $\tilde f_n$ can be
 generated by  $\tilde f_n(\Omega_n^0)$
 via the action of $G$, where
 $\Omega_n^0$ is the fundamental domain as in \eqref{eq:Omega0}.
\end{proposition}

\section{Properness of $\tilde f_n$}
Firstly, we prepare some inequalities which will be necessary for 
proving that $\tilde f_n$ is a proper mapping. 

By the definition \eqref{eq:omega} of $\Omega_n$, we have the 
following two inequalities on $\Omega_n$
(cf. \eqref{eq:meaningful}, \eqref{eq:tnu-cos}) 
\begin{equation}\label{eq:Phi-pos}
  T_n(u)>\cos n\theta, 
 \end{equation}
and
 \begin{equation}\label{eq:cos-pos}
  u>\cos\frac{\pi}{n} \text{ on }\Omega_n, 
 \end{equation}
since the function $\max_{j=0,\dots,n-1}
             \left[\cos\left(\theta-{2\pi j}/ n \right)\right]$ 
has a minimum value $\cos (\pi/n)$.

\begin{lemma}\label{lem:ut}
 On the fundamental domain $\Omega_n^0$,
 it holds that
 \begin{equation}\label{eq:positive}
   u - \cos\left(\theta-\frac{2\pi j}{n}\right)\geq 2\sin^2\frac{\pi}{n}
	    \qquad (j=2,\dots,n-1).
 \end{equation}
\end{lemma}
\begin{proof}
 Since $u>\cos\theta$ and $0\leq \theta\leq \pi/n$ on $\Omega_n^0$,
 \begin{align*}
   u - \cos\left(\theta-\frac{2\pi j}{n}\right)
    &>\cos\theta-\cos\left(\theta-\frac{2\pi j}{n}\right)
    = 2 \sin\left(\frac{\pi j}{n}-\theta\right)\sin\frac{\pi j}{n}\\
    &\geq 
     2 \sin\frac{\pi(j-1)}{n}\sin\frac{\pi j}{n}\geq 2\sin^2\frac{\pi}{n}
 \end{align*}
 for $2\leq j\leq n-1$, proving \eqref{eq:positive}.
\end{proof}

Using these, we prove the following assertion:
\begin{proposition}\label{prop:proper}
 The analytic extension $\tilde f_n \colon \Omega_n \to \R^3_1$ is 
a proper mapping.
\end{proposition}
\begin{proof}
By Proposition \ref{prop:symmetry},
it is sufficient to show that the restriction of 
$\tilde f_n$ to $\Omega_n^0$ is a proper mapping.
We set
 \begin{equation*}
    C:= \overline{\Omega_n^0}\setminus\Omega_n^0
      = \left\{(\cos\theta,\theta)\,;\,
	          0\leq \theta\leq \frac{\pi}{n}
               \right\}.
 \end{equation*}
Consider a 
sequence $\{(u_k,\theta_k)\}_{k=1,2,\dots}$
 in $\Omega_n^0$ 
such that
 \begin{equation*}
  \lim_{k\to\infty} (u_k,\theta_k)
   =(\cos\theta_{\infty},\theta_{\infty})\in C
   \qquad\left(0\leq\theta_{\infty}\leq\frac{\pi}{n}\right).
 \end{equation*}
It is sufficient to show that
the sequence
$\{\tilde f_n(u_k,\theta_k)\}$ is unbounded in $\R^3_1$.
 
\paragraph{Case 1:} 
 We consider the case that $0\leq \theta_{\infty} < \pi/n$.
The sequence $\{ (u_k, \theta_k) \}$ is bounded since it converges.
So we can take positive numbers $u_0$ and $\delta$ such that 
$u_k < u_0$ and $\theta_k < (\pi/n) - \delta$ for all $k$, that is, 

 \[
    (u_k,\theta_k) \in \Omega_{\delta,u_0}:=
       \left\{(u,\theta)\in \Omega_n^0\,;\,
 u\leq u_0, \ \theta\leq\frac{\pi}{n}-\delta\right\}.
 \]
 We now set (cf.\ \eqref{eq:u2})
 \begin{align*}
  \tilde x_2&= \tilde x_{2,a}+\tilde x_{2,b},\\
    &\tilde x_{2,a}:= \frac{-T_{n-1}(u)\cos \theta+u \cos (n-1)\theta}
               {n(T_n(u)-\cos n\theta)},\\
    &\tilde x_{2,b}:= 
      \frac{n-1}{n^2}
         \sum_{j=0}^{n-1}
             \log\left(
               u- \cos\left(\theta-\frac{2\pi  j}{n}\right)
          \right)
          \cos\frac{2\pi j}{n}.
 \end{align*}    
 Since the numerator of $\tilde x_{2,a}$ satisfies 
(cf. \eqref{eq:T} in the appendix)
 \begin{multline*}
         -T_{n-1}(u)\cos\theta+u \cos (n-1)\theta
\biggr |_{u=\cos\theta}
    = -\cos(n-1)\theta\cos\theta + \cos\theta \cos(n-1)\theta=0,
 \end{multline*}
 there exists a real analytic function $\varphi(u,\theta)$
 such that
 \begin{equation*}
   - T_{n-1}(u)\cos \theta + u \cos (n-1)\theta = (u-\cos\theta) \varphi(u,\theta).
 \end{equation*}
 Since
\begin{equation}\label{eq:positive1}
  \begin{aligned}
     u - \cos\left(\theta-\frac{2\pi}{n}\right)
     & \geq 
     \cos\theta - \cos\left(\theta-\frac{2\pi}{n}\right)\\
     &= 2\sin\left(\frac{\pi}{n}-\theta\right)\sin\frac{\pi}{n}
     > 2\sin\delta\sin\frac{\pi}{n}
 \end{aligned}
\end{equation}
 holds on $\Omega_{\delta,u_0}$,
\eqref{eq:tnu-cos} 
and  \eqref{eq:positive} in Lemma~\ref{lem:ut}
 yield that
 there exist a real analytic function $\psi(u,\theta)$ and a positive
 number $\varepsilon$ such that
 \begin{equation*}
    T_{n}(u)-\cos n\theta = (u-\cos\theta)\psi(u,\theta),\qquad
     \psi(u,\theta)\geq\varepsilon>0\quad\text{on}\quad\Omega_{\delta,u_0}.
 \end{equation*}
 Thus 
 $\tilde x_{2,a}={\varphi(u,\theta)}/{n\psi(u,\theta)}$
 is bounded on $\Omega_{\delta,u_0}$.

 Since
 \eqref{eq:positive} in Lemma~\ref{lem:ut} and \eqref{eq:positive1}
 imply that
 \[
     \log\left(u - \cos\left(\theta-\frac{2\pi j}{n}\right)\right)\qquad
     (j=1,2,\dots,n-1)
 \]
 is bounded on $\Omega_{\delta,u_0}$,
 we can write
 \begin{align*}
  \tilde x_{2,b}  
  = \frac{n-1}{n^2}\log(u-\cos\theta) + \beta(u,\theta),
 \end{align*}
 where $\beta(u,\theta)$ is a real analytic function 
bounded on $\Omega_{\delta,u_0}$.
Thus, 
 $\tilde x_2(u_k,\theta_k)\to-\infty$ as 
 $k\to \infty$. 

\bigskip
\paragraph{Case 2:}
 We next consider the case that
 $\theta_{\infty}=\pi/n$.
In other words, we suppose the sequence $\{ (u_k, \theta_k ) \}$ converges to 
$(\cos(\pi/n), \pi/n)$. 
In this case, 
we seek to prove
\begin{equation}\label{eq:limitx1}
\lim_{k\to\infty} \tilde x_1(u_k,\theta_k) = -\infty.
\end{equation}
 
We may assume $\{ (u_k, \theta_k ) \} \subset \Omega_n^0 \cap \{ u \le u_0 \}$ 
for some constant $u_0$. 
We set (cf. \eqref{eq:u1})
 \begin{align*}
  \tilde x_1 &= \tilde x_{1,a} + \tilde x_{1,b},\\
  &\begin{aligned}
      \tilde x_{1,a}:&=-\frac{T_{n-1}(u)\sin \theta+u \sin (n-1)\theta}{%
           n(T_n(u)-\cos n\theta)},\\
   \tilde x_{1,b}:&=
      \frac{n-1}{n^2}
       \sum_{j=1}^{n-1}
                        \log\left(
                 u- \cos\left(\theta-\frac{2\pi j}{n}
\right)\right)\sin\frac{2\pi j}{n}.
   \end{aligned}
 \end{align*}
Let $(u, \theta) \in \Omega_n^0$. Then $u > \cos \theta$ and 
$\cos \theta \in (\cos (\pi/n),1) \subset [\cos (\pi/n), \infty)$. 
So both $u$ and $\cos \theta$ 
belong to $[\cos (\pi/n), \infty)$. 
Since $T_{n-1}$ is monotone increasing on 
$[\cos (\pi/n), \infty) (\subset [\cos (\pi/(n-1)), \infty) )$
(cf. \eqref{eq:T} and Proposition \ref{prop:app:t} in the appendix), 
it holds that
 \[
    T_{n-1}(u) > T_{n-1}(\cos \theta) = \cos (n-1)\theta\qquad \text{on} \quad \Omega_{n}^0. 
 \]
 Noticing this,  
we have
 \[ 
    T_{n-1}(u)\sin\theta + u \sin(n-1)\theta
    \geq \cos (n-1)\theta \sin\theta + \cos\theta \sin (n-1)\theta
    =\sin n\theta\geq 0
 \]
 on $\Omega_n^0$.
By \eqref{eq:Phi-pos},
 the inequality
 $\tilde x_{1,a}\leq 0$ holds on $\Omega_{n}^0$.
 By \eqref{eq:positive} in Lemma~\ref{lem:ut},
 \[
     \log\left(u-\cos\left(\theta-\frac{2\pi j}{n}\right)\right)
     \qquad (j=2,\dots,n-1)
 \]
 is bounded on $\Omega_n^0 \cap \{ u \le u_0 \}$,
 and we can write
 \begin{align*}
   \tilde x_1 =
   \tilde x_{1,a} +\tilde x_{1,b} \leq \tilde x_{1,b}
   = \hat\beta(u,\theta) + 
                \log\left(u-\cos\left(\theta-
  \frac{2\pi}{n}\right)\right)\sin\frac{2\pi}{n}
 \end{align*}
 on $\Omega_n^0$, where $\hat \beta(u,\theta)$ 
 is a real analytic function bounded on $\Omega_n^0 \cap \{ u \le u_0 \}$.
 Since the right-hand side tends to $-\infty$
 as $(u,\theta)\to \bigl(\cos(\pi/n),\pi/n\bigr)$,
 \eqref{eq:limitx1} holds. 
\end{proof}

\section{%
Immersedness of $\tilde f_n$}
\label{sec:3}

\begin{proposition}\label{prop:immersion}
 The analytic extension $\tilde f_n \colon\Omega_n\to \R^3_1$
 is an immersion.
\end{proposition}
\begin{proof}
In this proof, $f, \tilde f$ denote $f_n, \tilde f_n$, respectively, 
for notational simplicity.

Since $\partial/\partial z = (1/2z)(r \partial/\partial r -
\imag \partial / \partial \theta)$ for $z=r e^{\imag \theta}$, 
we have
 \[
\alpha = dF = F_z dz = (F+\bar F)_z dz = 2f_z dz 
= \frac1{z}(r f_r-\imag f_\theta)dz, 
 \]
 where $\alpha=a(z)\,dz$ is as given in \eqref{eq:alpha},  
 that is,
 \begin{equation}\label{eq:key}
  z a(z) = r f_r-\imag f_\theta.
 \end{equation}
 On the other hand, we have (cf. \eqref{eq:alpha})
 \begin{align}\label{eq:za}
  z a(z)
  &= 
  \left(\frac{-2 \imag r^n e^{\imag n\theta}}
  {\left(r^n e^{\imag n\theta}-1\right)^2},
  \frac{\imag r e^{\imag \theta} 
  \left(1+r^{2n-2} e^{\imag(2n-2) \theta}\right)}
  {\left(r^n e^{\imag n\theta}-1\right)^2},\right.\\
  &\hspace{0.4\linewidth}\nonumber
  \left.-\frac{r e^{\imag \theta} \left(1-
  r^{2n-2} e^{\imag(2n-2) \theta}\right)}
  {\left(r^n e^{\imag n\theta}-1\right)^2}\right).
 \end{align}
 We define by $\xi^{jk}:=\trans{(\xi^j,\xi^k)}$
 ($j<k$)
 for $\xi=(\xi^0,\xi^1,\xi^2)\in \C^3$, and
 here $\trans{(*)}$ means transposition.
 Then 
 \[
 2 \imag r \det(f_r^{jk},f_\theta^{jk})
       =\det(r f_r^{jk}-\imag f_\theta^{jk},
           r f_r^{jk}+\imag f_\theta^{jk})
       =
         \det(z a^{jk},\overline{z a^{jk}}).
 \]
 Using this and \eqref{eq:za},
 one can arrive at
 \begin{align}\label{eq:12}
  \det(f_r^{01},f_\theta^{01})&=
  \frac{2r^{n-2} \left(r^{2 n}-r^2\right) \sin (n-1)\theta}
  {\left(r^{2 n}-2 r^n \cos n \theta+1\right)^2},\\
  \label{eq:13}
  \det(f_r^{02},f_\theta^{02})&=
  -\frac{2r^{n-2} \left(r^{2 n}-r^2\right) \cos (n-1)\theta}
  {\left(r^{2 n}-2 r^n \cos n \theta+1\right)^2}.
 \end{align}
 Since we set $u=(r+r^{-1})/2$, we have
 \begin{equation}\label{eq:fu}
  f_u=\frac{2r^2}{r^2-1}f_r.
 \end{equation}
 The equality \eqref{eq:12} is equivalent to
 \begin{align*}
  \det(f^{01}_u,f^{01}_\theta)
  &=
  \frac{4r^{n} \left(r^{2 n}-r^2\right) \sin (n-1)\theta}
  {(r^2-1)\left(r^{2 n}-2 r^n \cos n \theta+1\right)^2}\\
  &=
  \frac{4\left(r^{n-1}-r^{1-n}\right) \sin (n-1)\theta}
  {(r-r^{-1})\left(r^{n}+r^{-n}-2 \cos n \theta\right)^2}
  =
  \frac{U_{n-2}(u) \sin (n-1)\theta}
  {\left(T_n(u)-\cos n \theta\right)^2}, 
 \end{align*}
where $U_{n-2}(u)$ denotes the second Chebyshev polynomial of degree $n-2$. 
(See \eqref{eq:T2}, \eqref{eq:U} and \eqref{eq:U2} in the appendix.)
 Similarly, by \eqref{eq:13}, we have
 \[
    \det(f^{02}_u,f^{02}_\theta)
            =-\frac{U_{n-2}(u) \cos (n-1)\theta}
             {\left(T_n(u)-\cos n \theta\right)^2}.
 \]
 By the real analyticity, the identities
 \begin{equation}\label{eq:two}
  \det(\tilde f^{01}_u,\tilde f^{01}_\theta)
   =
   \frac{U_{n-2}(u) \sin (n-1)\theta}
   {\left(T_n(u)-\cos n \theta\right)^2},\ 
   \det(\tilde f^{02}_u,\tilde f^{02}_\theta)
   =-\frac{U_{n-2}(u) \cos (n-1)\theta}
   {\left(T_n(u)-\cos n \theta\right)^2}
 \end{equation}
 hold on $\Omega_n$.
Hence, it cannot occur that $\det(\tilde f^{01}_u,\tilde f^{01}_\theta)$ and 
$\det(\tilde f^{02}_u,\tilde f^{02}_\theta)$ vanish simultaneously,   
 since $U_{n-2}(u) >0$ by \eqref{eq:cos-pos}  
 (cf. Corollary \ref{cor:app:u} in the appendix).
 We conclude that
 $\tilde f_n$ is an immersion.
\end{proof}

\section{Embeddedness of $\tilde f_n$}
\label{sec:fin}
\subsection{Outline}
We show  that $\tilde f_n=(\tilde x_0, \tilde x_1, \tilde x_2)
\colon \Omega_n \to (\R^3_1 ; t,x,y)$ is an embedding.
The set
$$
\tilde x_0^{-1}(h)
$$
is called the {\it contour-line} 
of height $t=h$, and
$$
\Lambda_h :
=\tilde f_n(\tilde x_0^{-1}(h))
=\tilde f_n (\Omega_n) \cap \{ t=h \}
$$ 
is called the {\it level curve set} of 
height $t=h$.
To show the embeddedness of $\tilde f_n$,
it is sufficient to 
show that $\tilde f_n \colon 
\tilde x_0^{-1}(h) \to \Lambda_h$ 
is injective at each height $h$.  

Since (cf. \eqref{eq:x0} or \eqref{eq:u0}) 
\begin{equation}
 \label{eq:u02}
 \tilde x_0 = \frac{2 r^n \sin n \theta}{n(r^{2n} -2 r^n \cos n \theta +1)}
= \dfrac{\sin n \theta}{n(T_n(u)-\cos n \theta)}, 
\end{equation}
the contour-line of height $h=0$ is given by
\begin{equation}\label{eq:x0-10}
\tilde x_0^{-1}(0) = 
 \bigcup_{k=0}^{2n-1}\left\{ (u, \theta) \in \Omega_n \, ; \,  
\theta= \frac{k}{n}\pi \right\} 
\cup \{  p_\infty \}. 
\end{equation}
The figure of the contour-line
$\tilde x_0^{-1}(0)$ 
and its image (i.e. the level curve set of height $h=0$)
are indicated in Figure \ref{fg:cl-lcs-n5}.

\begin{figure}[htbp] 
\begin{center}
\begin{tabular}{ccc}
 \includegraphics[width=.45\linewidth]{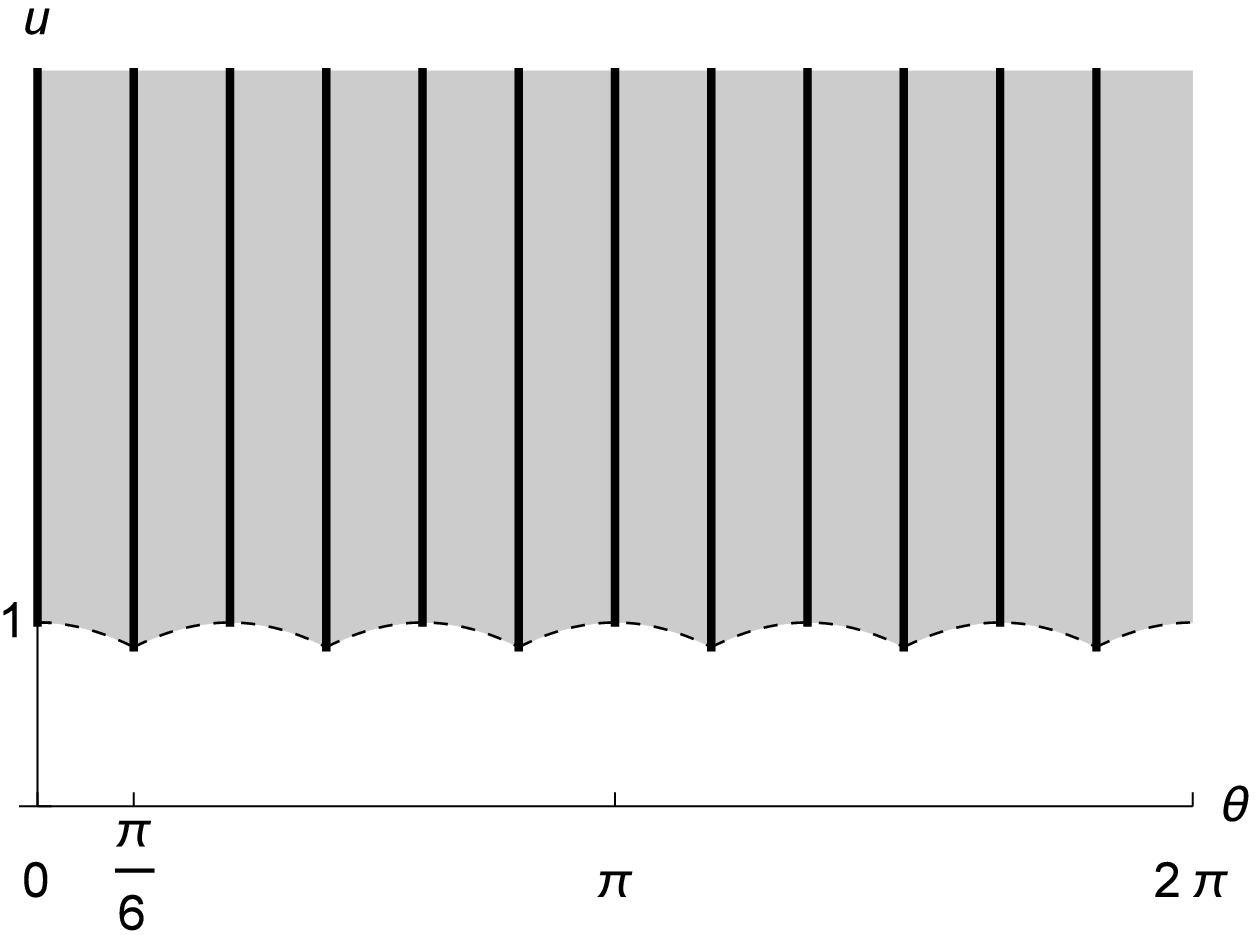} & 
\raisebox{35pt}[0pt]{$\xrightarrow{\quad \tilde f_6 \quad}$} & 
 \includegraphics[width=.34\linewidth]{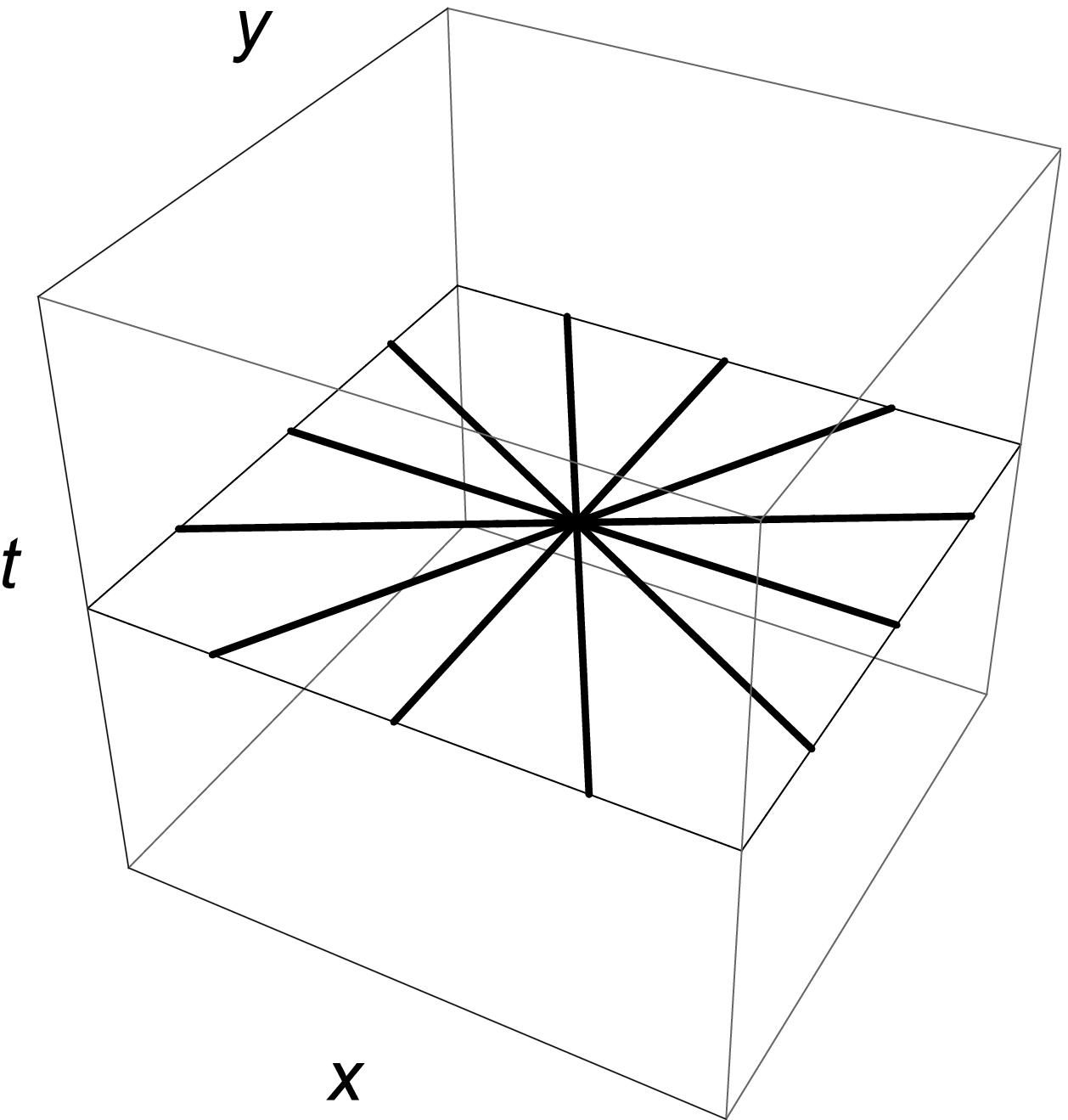} 
\end{tabular}
\caption{Contour-lines $\tilde x_0^{-1}(0)$ and
the level curve set $\Lambda_0$ at height $h=0$ in the case $n=6$.}
\label{fg:cl-lcs-n5}
\end{center}
\end{figure} 

On the other hand, if $h\ne 0$, we have
\begin{equation*}
\tilde x_0^{-1}(h) = 
 \left\{ (u, \theta) \in \Omega_n \, ; \,  
T_n(u) = \cos n \theta + \frac{1}{n h} \sin n \theta \right\}.  
\end{equation*}
The following assertion is immediately obtained.
\begin{proposition}\label{prop:nonneg}
 The function $\tilde x_0$ 
$($cf.  \eqref{eq:u02}$)$
is non-negative valued
 on $\Omega_n^0$, where
$\Omega_n^0$ is the fundamental domain given by
\eqref{eq:Omega0}.
\end{proposition}
Since
$\tilde x_0^{-1}(h) \cap \Omega_n^0 = \emptyset$ if $h<0$, 
we may suppose $h > 0$ to prove the injectivity of
$\tilde f_n:\tilde x_0^{-1}(h)\to \Lambda_h$
for $h\ne 0$.
Let $\Lambda_h^0$ be the level curve set of the image 
$\tilde f_n(\Omega_n^0)$ of the fundamental domain $\Omega_0$, 
that is, 
\begin{equation*}
 \Lambda_h^0 := \tilde f_n(\Omega_n^0) \cap \{ t=h \} 
= \tilde f_n \left( \tilde x_0^{-1}(h) \cap \Omega_n^0 \right).
\end{equation*}
As a consequence of Proposition \ref{prop:G}, 
we obtain the following:
\begin{corollary}\label{cor:lcs-f}
\begin{equation*}
 \Lambda_h = \bigcup_{k=0}^{n-1} R^k \Lambda_h^0,\quad
\mbox{and}\quad
 \Lambda_{-h} = S \Lambda_h, 
\end{equation*}
where 
$R$ and $S$ are the matrices defined in \eqref{eq:SR} (cf. \eqref{eq:1}).
\end{corollary}
Corollary \ref{cor:lcs-f} implies that we should seek to prove that 
\begin{enumerate}
 \item the map $\tilde f_n$ restricted to $\tilde x_0^{-1}(h) \cap \Omega_n^0$, i.e.,  
$\tilde f_n \colon \tilde x_0^{-1}(h) \cap \Omega_n^0 \to \Lambda_h^0$ is 
injective,
 \item $\bigcup_{k=0}^{n-1} R^k \Lambda_h^0$ is a disjoint union.  
\end{enumerate}

\begin{figure}[htbp] 
\begin{center}
\begin{tabular}{ccc}
 \includegraphics[width=.30\linewidth]{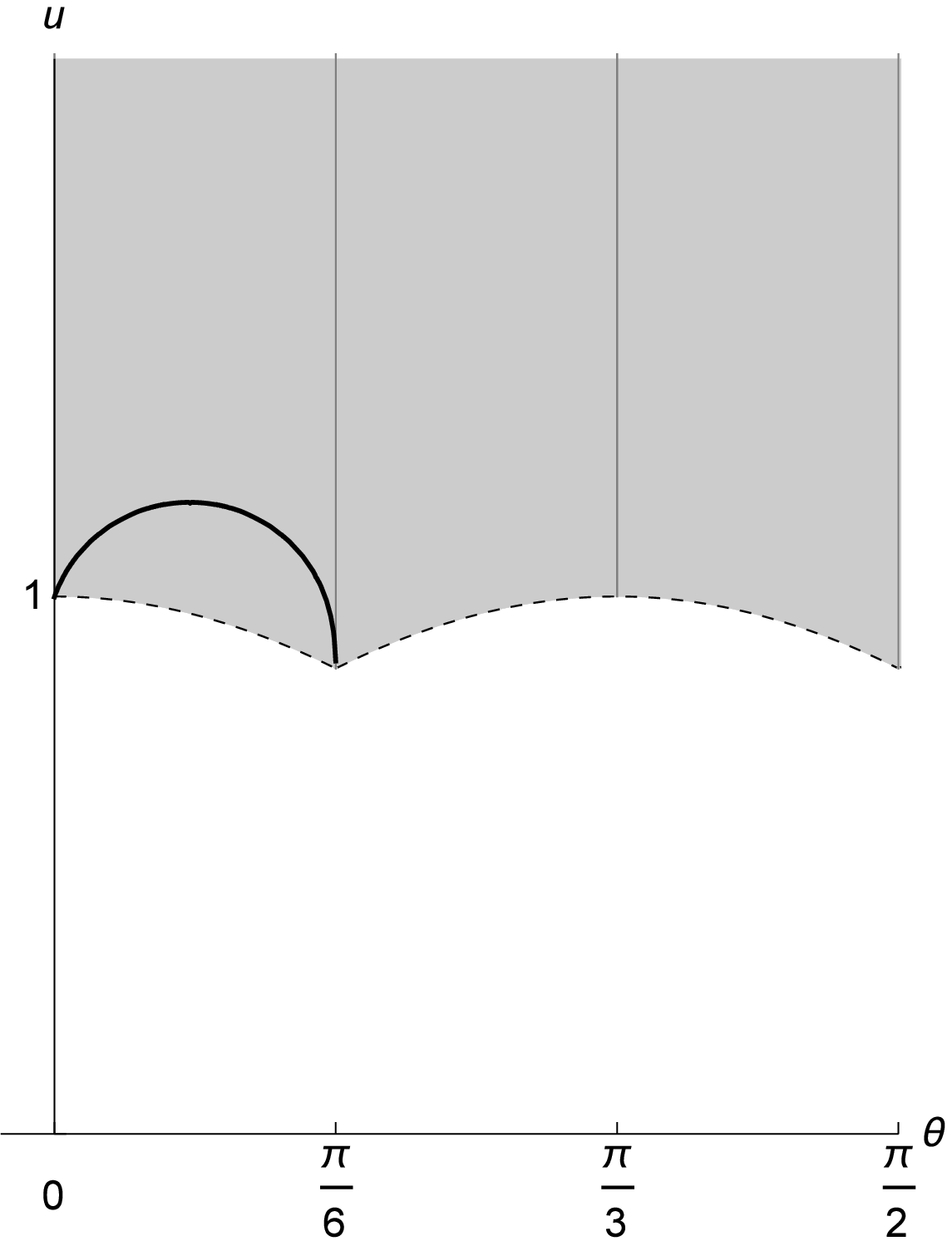} & 
\raisebox{35pt}[0pt]{$\xrightarrow{\quad \tilde f_6 \quad }$} & 
 \includegraphics[width=.34\linewidth]{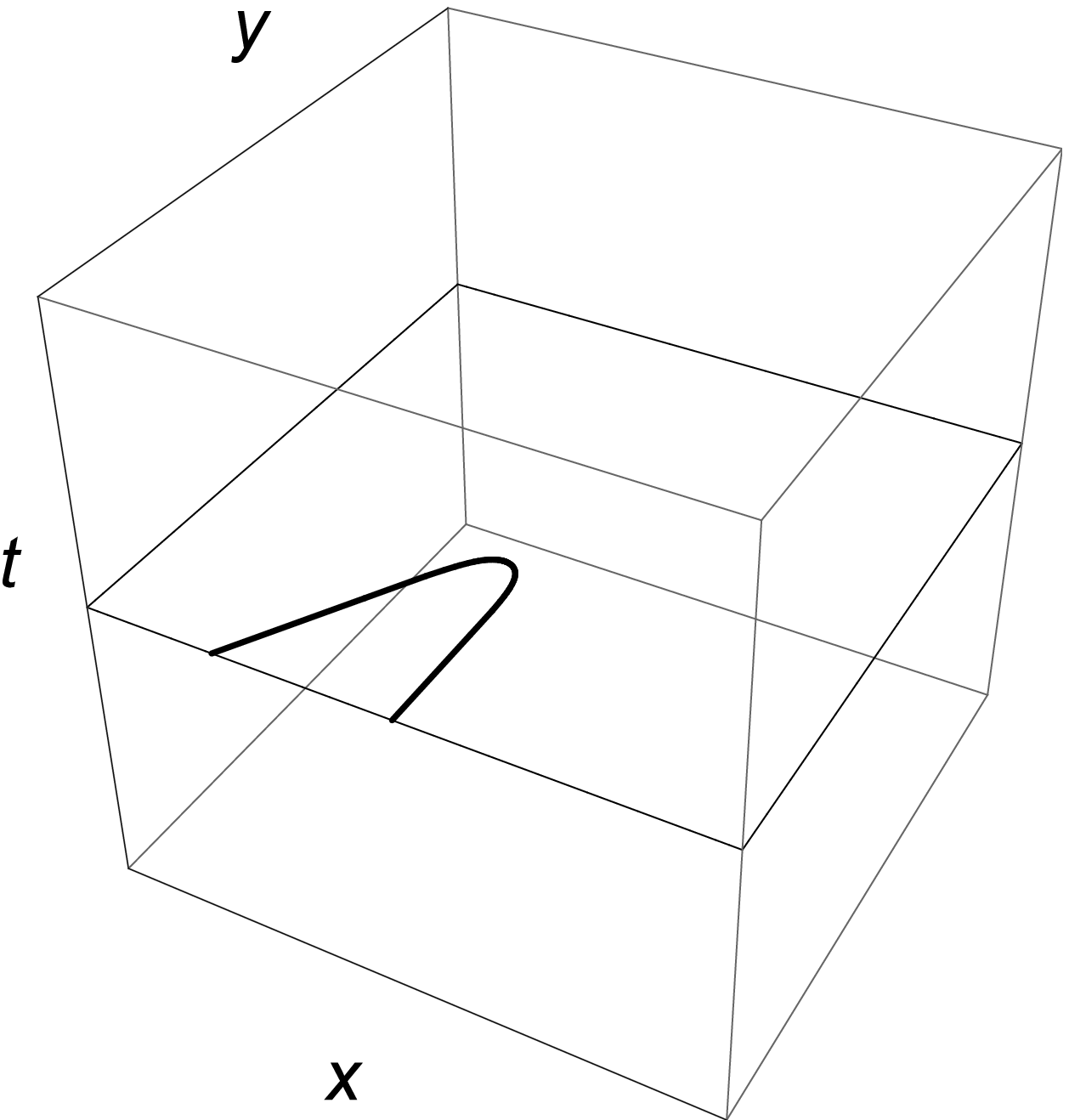} 
\end{tabular}
\caption{Contour-line $\tilde x_0^{-1}(h) \cap \Omega_n^0$ and
the level curve set $\Lambda_h^0$ for $h=0.01$ in the case $n=6$.}
\label{fg:cl-lcs-6-0}
\end{center}
\end{figure} 

\bigskip

\begin{figure}[thbp] 
\centering
 \includegraphics[width=.33\linewidth]{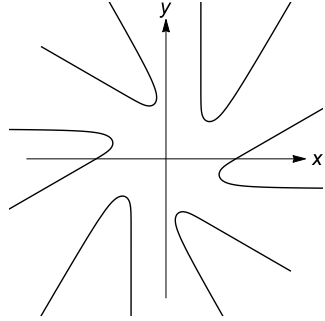} 
\caption{The level curve set $\Lambda_h$ of $\tilde f_{6}(\Omega_6)$ 
for $h=1$}
\label{fg:hyp-cyl}
\end{figure}

\subsection{Contour-lines in $\Omega_n^0$}
Now we investigate $\tilde x_0^{-1}(h) \cap \Omega_n^0$. As mentioned 
above, 
we suppose $h>0$. 

\begin{proposition}\label{x0:h}
\begin{enumerate}[{\rm(1)}]
 \item \label{item:implicit-ft}%
Given $h >0$ and $0 < \theta < \pi/n$, the equation $\tilde x_0 (u, \theta)=h$ 
is uniquely solved for $u \in (\cos (\pi/n), \infty)$. Indeed, it 
determines the implicit function $u=u(h, \theta)$ 
defined on $A:=\{ (h, \theta) \mid h>0, \ 0< \theta < \pi/n\}$ satisfying 
$\cos \theta < u(h, \theta)$. Moreover, the following hold: 
\begin{enumerate}
 \item For a fixed $\theta_0 \in (0, \pi/n)$, the function 
$h \mapsto u(h, \theta_0)$ is monotone decreasing, and 
\begin{equation*}
 \lim_{h \searrow 0} u(h, \theta_0) = \infty.
\end{equation*}
 
 \item For a fixed $h_0 >0$, 
\begin{equation*}
 \lim_{\theta \searrow 0} u(h_0, \theta) = 1, \quad 
 \lim_{\theta \nearrow {\pi}/n } u(h_0, \theta) = \cos (\pi/n). 
\end{equation*}
\end{enumerate}
\item \label{item:x0u}
The derivative $(\tilde x_0)_u$ is given by 
\begin{equation*} 
(\tilde x_0)_u =  -\frac{U_{n-1}(u) \sin n \theta}
  {\left(T_n(u)-\cos n \theta\right)^2}. 
\end{equation*}
\end{enumerate}
\end{proposition}
\begin{proof}
(\ref{item:implicit-ft}): The equation $\tilde x_0 (u, \theta)=h (>0)$ 
is equivalent to 
$$T_n (u) = \cos n \theta + \frac{1}{nh} \sin n \theta,$$ 
and 
$$-1 < \cos n \theta < \cos n \theta + \frac{1}{nh} \sin n \theta$$ 
holds 
on $A=\{ (h, \theta) \mid h>0, \ 0< \theta < \pi/n \}$.   
On the other hand,  
the Chebyshev polynomial $T_n(u)$ is monotone increasing on
the interval $[\cos (\pi/n) , \infty)$ 
(see Proposition \ref{prop:app:t} in the appendix), 
and hence it has the inverse function 
$$
T_n^{-1} \colon [-1, \infty) \to [\cos (\pi/n), \infty),
$$
which is monotone increasing. Thus, 
\begin{equation}\label{eq:explicit-u}
 u(h, \theta):=T_n^{-1} ( \cos n \theta + \frac{1}{nh} \sin n \theta )
\quad \text{on $A$} 
\end{equation}
is well-defined and the desired one.  
Obviously, $$\cos \theta = T_n^{-1}(\cos n \theta) 
< T_n^{-1} ( \cos n \theta + \frac{1}{nh} \sin n \theta ) = u(h, \theta)$$ 
holds on $A$. 
 
Since $T_n^{-1}$ is monotone increasing on $[-1, \infty)$, the formula  
\eqref{eq:explicit-u} immediately implies the assertions (i) and (ii).   

\medskip

\noindent
(\ref{item:x0u}): 
This can be determined directly from \eqref{eq:u0}. 
\end{proof}

Hereafter, we set (cf. \eqref{eq:explicit-u})
$$
u_h(\theta):=u(h,\theta)
$$   
which can be considered as a function of $\theta$
fixing $h$.
Proposition \ref{x0:h} implies that the contour-line 
$\tilde x_0^{-1}(h) \cap \Omega_n^0$ satisfies 
\begin{equation}\label{eq:con-lin}
\tilde x_0^{-1}(h) \cap \Omega_n^0
= \left\{ (u, \theta) \in \Omega_n^0 \mid u=u_h(\theta) \right\}
= \left\{ (u_h(\theta), \theta) \in \Omega_n^0 \mid 0<\theta <\pi/n \right\}.  
\end{equation}
The level curve set $\Lambda_h^0$, i.e.,
$\tilde f_n (\tilde x_0^{-1}(h) \cap \Omega_n^0)$ is given by
\begin{equation*}
\Lambda_h^0=
\left\{ (h, \tilde x_1(u_h(\theta),\theta), \tilde x_2(u_h(\theta),\theta)) \mid 
0 < \theta < \pi/n \right\}.
\end{equation*}
We show the 
following properties of the level curve set $\Lambda_h^0$.

\begin{lemma}\label{lem:x2}\

\begin{enumerate}
  \item\label{item:mon3} 
       $\tilde x_1(u_h(\theta),\theta)$ 
       is a monotone decreasing function of 
       $\theta\in (0,\pi/n)$, whose value is less than $-h$, 
  \item\label{item:mon4} 
       $\tilde x_2(u_h(\theta),\theta)$ attains a maximum at  
       $\theta= \dfrac{\pi}{2(n-1)} \in (0,\pi/n)$.
 \end{enumerate}
\end{lemma}
\begin{proof}
\ref{item:mon3}:
 By \eqref{eq:two} and Proposition \ref{x0:h} \eqref{item:x0u}, 
we have 
\begin{equation}\label{eq:x1th}
  \begin{aligned}
   \frac{d}{d\theta}\tilde x_1\bigl(u_h(\theta),\theta\bigr)
  &= 
     \frac{\partial \tilde x_1}{\partial u}\frac{d u_h}{d\theta}+
    \frac{\partial \tilde x_1}{\partial \theta} 
   =
    - \frac{\partial \tilde x_1}{\partial u}
      \frac{\frac{\partial
      \tilde x_0}{\partial\theta}}
     {\frac{\partial \tilde x_0}{\partial u}
         }+
    \frac{\partial \tilde x_1}{\partial \theta} 
\\
  &=
  \frac{1}{(\tilde x_0)_u}\det(\tilde f^{01}_u,
  \tilde f^{01}_\theta)
= - \frac{U_{n-2}(u_h(\theta))}{U_{n-1}(u_h(\theta))} 
\frac{\sin (n-1) \theta}{\sin n \theta},
 \end{aligned}
\end{equation}
which is negative for $0 < \theta < \pi/n$ (cf. Corollary \ref{cor:app:u} in the 
appendix). 
Hence, $\tilde x_1 (u_h(\theta), \theta)$ is a 
monotone decreasing function of $\theta$. 
Next, according to \eqref{eq:u1}, we set 
\begin{align*}
 \tilde x_1 (u_h(\theta), \theta) = \tilde x_{1,a}(\theta) + \tilde x_{1,b}(\theta) , 
\end{align*}
where
\begin{align*}
 \tilde x_{1,a} (\theta) &= -h \left( T_{n-1}(u_h(\theta))\frac{\sin \theta}{\sin n \theta}
 +  u_h(\theta) \frac{\sin (n-1)\theta}{\sin n \theta} \right), \\
 \tilde x_{1,b} (\theta) & = \frac{n-1}{n^2} \sum_{j=1}^{n-1}
\log \left(  u_h(\theta)- \cos \left( \theta - \frac{2 \pi j}{n}\right)\right)
\sin \frac{2 \pi j}{n} .
\end{align*} 
These satisfy 
\begin{align*}
 \lim_{\theta \searrow 0} \tilde x_{1,a} (\theta) = 
-h \left( T_{n-1}(1) \frac{1}{n} + 1 \frac{n-1}{n} \right)
= -h \left( \frac{1}{n} +\frac{n-1}{n} \right) = -h,  
\end{align*}
because of part (ii) of item (\ref{item:implicit-ft}) in 
Proposition \ref{x0:h}. Moreover, 
\begin{align*}
 \lim_{\theta \searrow 0} \tilde x_{1,b} (\theta) =  
\frac{n-1}{n^2} \sum_{j=1}^{n-1}
\log \left( 1- \cos \left( 0 - \frac{2 \pi j}{n}\right)\right)
\sin \frac{2 \pi j}{n} =0 
\end{align*}
holds, since the terms in the summation cancel 
for each pair $(j,n-j)$. 
Therefore 
\begin{equation*}
 \lim_{\theta \searrow 0} \tilde x_1 (u_h(\theta), \theta) = -h+0 =-h. 
\end{equation*}
Since the function $\theta \mapsto \tilde x_1 (u_h(\theta), \theta)$ 
is monotone decreasing, 
we conclude that 
$\tilde x_1 (u_h(\theta), \theta) < -h$ for all 
$\theta \in (0, \pi/n)$.

\medskip

\noindent
\ref{item:mon4}: 
 Similarly to \eqref{eq:x1th}, we have
\begin{equation}\label{eq:x2th}
 \begin{aligned}
   \frac{d}{d\theta}\tilde x_2\bigl(u_h(\theta),\theta\bigr)&=
     \frac{\partial \tilde x_2}{\partial u}\frac{d u_h}{d\theta}+
    \frac{\partial \tilde x_2}{\partial \theta}
   =
    - \frac{\partial \tilde x_2}{\partial u}\frac{\frac{\partial
 \tilde x_0}{\partial\theta}}{\frac{\partial \tilde
 x_0}{\partial u}
 }+
    \frac{\partial \tilde x_2}{\partial \theta}
\\
 &=
   \frac{1}{(\tilde x_0)_u}
         \det(\tilde f^{02}_u,\tilde f^{02}_\theta)
= \frac{U_{n-2}(u_h(\theta))}{U_{n-1}(u_h(\theta))} 
\frac{\cos (n-1) \theta}{\sin n \theta} , 
 \end{aligned}
\end{equation}
 which is
\begin{equation*}
 \begin{cases}
 \text{positive if }\,\,  0 < \theta < \pi/2(n-1), \\
 \text{zero if }\,\, \theta = \pi/2(n-1), \\
 \text{negative if }\,\, \pi/2(n-1) < \theta < \pi/n .
\end{cases}    
\end{equation*}
This proves the assertion \ref{item:mon4}.
\end{proof}

\begin{proposition}\label{prop:fn-inj}
The restriction of the map $\tilde f_n$
given by
\begin{equation}\label{eq:fn-restricted}
 \tilde f_n \colon 
 \tilde x_0^{-1}(h) \cap \Omega_n^0 \ni (u_h(\theta), \theta) \mapsto 
(h, \tilde x_1(u_h(\theta), \theta), \tilde x_2(u_h(\theta)) \in \Lambda_h^0 
\end{equation} 
is injective.
\end{proposition}
\begin{proof}
\eqref{eq:con-lin} and Lemma \ref{lem:x2} (1) imply that 
the above correspondence \eqref{eq:fn-restricted}
gives a regular curve without self-intersection. 
\end{proof}

\subsection{Level curve sets}
Firstly, we deal with the level curve set $\Lambda_0$ of height $h=0$
(cf. Figure \ref{fg:cl-lcs-n5}).  

\begin{proposition}\label{prop:cl-0}
 The map $\tilde f_n$ restricted to $\tilde x_0^{-1}(0)$ is injective. 
\end{proposition}

To prove the assertion, we prepare the following lemma:

\begin{lemma}
 \begin{align}
  (\tilde x_1)_u
  &=\frac{\sin (2n-1)\theta + 2 U_{n-2}(u)\sin (n-1)\theta +U_{2n-2}(u)\sin \theta}
  {2(T_n(u)-\cos n \theta)^2},
  \label{al:x1u} \\
  (\tilde x_2)_u
  &= \frac{-\cos (2n-1)\theta - 2 U_{n-2}(u)\cos (n-1)\theta +U_{2n-2}(u)\cos \theta}
  {2(T_n(u)-\cos n \theta)^2}.
   \label{al:x2u}
  \end{align}
\end{lemma}
\begin{proof}
 By \eqref{eq:key}, \eqref{eq:za} and 
 \eqref{eq:fu}, we obtain that
 \begin{align}
  \frac{(\tilde x_1)_u}2 &=
  \frac{r}{r^2-1}\Re(za_1(z)) \notag \\
  &=\frac{2 (r^{3 n}-r^{n+2})\sin (n-1) \theta
  +\left(r^2-1\right) r^{2 n} \sin (2n-1)\theta
  +\left(r^{4 n}-r^2\right) \sin \theta}
  {\left(r^2-1\right) 
  \left(r^{2 n}-2 r^n \cos n \theta+1\right)^2}, \notag \\
  \frac{(\tilde x_2)_u}2
  &=\frac{r}{r^2-1}\Re(za_2(z)) \notag \\
  \nonumber
  &=-\frac{2 (r^{3 n}-r^{n+2}) \cos (n-1)\theta
  +\left(r^2-1\right) r^{2 n} 
  \cos (2n-1)\theta+\left(r^2-r^{4 n}\right) \cos \theta}
  {\left(r^2-1\right) \left(r^{2 n}-2 r^n \cos n \theta
  +1\right)^2}.
  \end{align}
So $(r+r^{-1})/2 =u$ proves \eqref{al:x1u} and \eqref{al:x2u}.
\end{proof}

\begin{proof}[Proof of Proposition \ref{prop:cl-0}]
Recall the equality \eqref{eq:x0-10} which asserts that 
\begin{equation*}
 \tilde x_0^{-1}(0) = \bigcup_{k=0}^{2n-1} B_k \cup  \{ p_\infty \},  
\end{equation*}
where
\begin{equation*}
 B_k := \left\{ (u, \theta) \in \Omega_n \mid 
\theta= \frac{k}{n}\pi \right\}.  
\end{equation*}
Consider the map 
\begin{equation*}
 \tilde f_n |_{B_k}=(\tilde x_0, \tilde x_1, \tilde x_2)|_{\theta = k \pi/n}
= (0, \tilde x_1(u,k \pi/n), \tilde x_2(u,k \pi/n)). 
\end{equation*} 
It follows from \eqref{al:x1u}, \eqref{al:x2u} 
and \eqref{eq:kaw}
that 
\begin{align*}
 \frac{d}{du}\left( \left. \tilde f_n \right|_{B_k} \right)
= V(u)
\left(0, \sin \frac{k}{n}\pi, \cos \frac{k}{n}\pi \right), 
\end{align*}
where
$$
V(u):=\frac{U_{n-2}(u)}{T_n(u)-(-1)^k}.
$$
This implies that $\tilde f_n |_{B_k}$ parametrizes
a straight half-line with the velocity $V(u)$.
If $k$ is even, $\tilde f_n |_{B_k}$ is 
defined on the interval $(1, \infty)$ 
and $V(u)$ is positive on $(1, \infty)$. 
If $k$ is odd, $\tilde f_n |_{B_k}$ is defined on the interval $(\cos (\pi/n), \infty)$ 
and $V(u)$ is positive on $(\cos (\pi/n), \infty)$.
Hence, for any $k$, the map $\tilde f_n|_{B_k}$ is injective. Moreover,  
the monotonicity of $\tilde f_n|_{B_k}$ and the equality 
\begin{equation*}
 \lim_{u \to \infty}\tilde f_n|_{B_k}(u) = \tilde f_n ( p_\infty ) = (0,0,0)
\end{equation*}  
 imply that the point $p_\infty$ is the  
unique inverse-image of $(0,0,0)$. 
Therefore we conclude the map 
$\tilde f_n \colon \tilde x_0^{-1}(0) \to \R^3_1$ is injective. 
\end{proof}

\medskip 
We next consider the case where the height $h$ is not equal to $0$.

For a fixed $h$, let $P_h$ denote a plane in $(\R^3_1;t,x,y)$ defined 
by the equation $t=h$, with coordinate system $(x,y)$.

\begin{proposition}\label{prop:lcs}
For any fixed $h>0$, the level curve set $\Lambda_h^0$ of height $h$ 
lies in the region 
\begin{equation*}
D_h:= \left\{  (x, y) \mid 
x < -h, \ x \cos(2\pi/n) -y \sin (2 \pi/n) + h >0 \right\} \subset P_h. 
\end{equation*}
\end{proposition}
\begin{proof}
 We parametrize $\Lambda_h^0$ so that 
(cf.  Proposition \ref{prop:fn-inj}) 
$$(x_h(\theta), y_h(\theta)) := 
\left( \tilde x_1 (u_h(\theta), \theta), \tilde x_2 (u_h(\theta), \theta) \right) \  
(0<\theta < \pi/n).
$$ 
We have already determined that $x_h(\theta) <-h$ 
(cf. Lemma \ref{lem:x2} \ref{item:mon3}). 
It remains to show 
\begin{equation*}
\phi_h(\theta) := x_h(\theta) \cos (2 \pi/n) - y_h (\theta) \sin (2 \pi/n) +h >0 
\end{equation*}
 for $0<\theta < \pi/n$. 
 Using \eqref{eq:x1th}, \eqref{eq:x2th},  we have 

\begin{equation*}
 \frac{d}{d \theta} \phi_h(\theta)
= 
-\frac{U_{n-2}(u_h(\theta))}{U_{n-1}(u_h(\theta))\sin n \theta}
\sin \left( (n-1) \theta + \frac{2 \pi}{n} \right).
\end{equation*}
This implies that $\phi_h(\theta)$ has a 
minimum at 
$$\theta_0 = \dfrac{n-2}{(n-1)n} \pi. $$ 
Hence, we have only to prove that
\begin{equation}\label{eq:keyofproof}
  \phi_h(\theta_0) >0 . 
\end{equation}
Indeed,
\begin{equation*}
\Phi(h) :=  \phi_h(\theta_0) = 
\tilde x_1( u_h(\theta_0),\theta_0)\cos(2 \pi/n)
- \tilde x_2( u_h(\theta_0),\theta_0)\sin(2 \pi/n)+h
\end{equation*}
satisfies
\begin{equation}\label{eq:limPhi}
 \lim_{h \searrow 0} \Phi(h) = 0 \cdot \cos(2 \pi/n) -0 \cdot \sin(2 \pi/n) + 0 =0,
\end{equation}
because of part (i) of item (\ref{item:implicit-ft}) in Proposition \ref{x0:h}. Moreover, a straightforward
computation using Proposition \ref{x0:h} (\ref{item:x0u}), 
\eqref{al:x1u} and \eqref{al:x2u} 
leads us to 
\begin{equation*}
\frac{d \Phi}{d h}(h) = \frac{1+U_{2n-2}(u_h(\theta_0))}{2 U_{n-1}(u_h(\theta_0))}+1
= \frac{1+U_{2n-2}(u_h(\theta_0))+ 2 U_{n-1}(u_h(\theta_0))}{2 U_{n-1}(u_h(\theta_0))}.
\end{equation*}
We wish to know the sign of $d \Phi/d h$. 
Note that $u_h(\theta_0) \in (\cos \theta_0, \infty)$ for  
$h \in (0, \infty)$. For this purpose, we set
\begin{equation*}
\Upsilon(u):= \frac{1+U_{2n-2}(u)+ 2 U_{n-1}(u)}{2 U_{n-1}(u)}
\text{ for } u \in (\cos \theta_0, \infty). 
\end{equation*}
Then, it is obvious 
$\Upsilon(u) >0$ for $u \in [1, \infty)$  
(cf. Proposition \ref{prop:app:u} in the appendix).
For $u \in (\cos \theta_0, 1)$, it is also obvious that 
the denominator of $\Upsilon(u)$ is positive. 
Since there exists a unique $\alpha \in (0, \theta_0)$
with $u = \cos \alpha$,  
the numerator is computed as
\begin{multline*}
 1+U_{2n-2}(\cos \alpha)+ 2 U_{n-1}(\cos \alpha) = 
\frac{\sin \alpha + \sin (2n-1) \alpha +2\sin n \alpha}{\sin \alpha} \\
= \frac{2 \sin n \alpha \cos (n-1) \alpha + 2 \sin n \alpha}{\sin \alpha} 
= \frac{2 \sin n \alpha}{\sin \alpha}(\cos (n-1)\alpha +1).  
\end{multline*}
So the numerator is positive because $0<\alpha < \theta_0= \frac{n-2}{(n-1)n}\pi$. 

Thus, $\Upsilon(u) >0$ for all $u \in (\cos \theta_0, \infty)$.  
Hence we obtain  
\begin{equation}\label{eq:Phi'-posi}
  \frac{d \Phi}{d h}(h) =\Upsilon(u_h(\theta_0)) >0 \text{ for } h \in (0,\infty).
\end{equation}

It follows from \eqref{eq:limPhi} and \eqref{eq:Phi'-posi} that
$\Phi(h) >0$ for all $h \in (0, \infty)$, that is, 
$\phi_h(\theta_0) >0$ for all $h \in (0, \infty)$. 
We have now proved \eqref{eq:keyofproof}. 
\end{proof}

\bigskip

We are in a position to complete a proof of the embeddedness of 
$\tilde f_n$.  

\begin{theorem}
 For any integer $n \ge 2$, the analytic extension $\tilde f_n:\Omega_n\to
 \R^3_1$ 
 is a proper embedding. 
\end{theorem}

\begin{proof}
The assertion for $n=2$ is trivial, as stated in Section \ref{sec:1}.  
 We have already proved that
 $\tilde f_n$ is a proper immersion
 (cf.\ Propositions 
\ref{prop:proper} and
\ref{prop:immersion}).
So it is sufficient to show that
$\tilde f_n$ is injective for each $n \ge 3$. 
For this purpose, we will show that $\tilde f_n$ restricted to each contour-line 
$\tilde x_0^{-1}(h)$ 
is injective. We have already done this for $h=0$ in Proposition \ref{prop:cl-0}.  
For $h \ne 0$, it suffices to show $\Lambda_h^0$ never intersects 
the other $R^k \Lambda_h^0$ ($k=1,2,\dots , n-1$), since we have 
already seen  
$\tilde f_n \colon \tilde x_0^{-1}(h) \cap \Omega_n^0 \to 
\Lambda_h^0 \subset D_h$  is injective (cf. Proposition \ref{prop:fn-inj}). 
In fact, 
the region $D_h$ of Proposition \ref{prop:lcs} does not intersect 
the other $R^k(D_h)$ ($k=1,2,\dots , n-1$) (see Figures \ref{fig:region-6-h1} and 
\ref{fig:region-6-h2}), thus, 
$\Lambda_h^0$ never intersects the other $R^k \Lambda_h^0$.   
Therefore, we conclude that $\tilde f_n \colon \Omega_n \to \R^3_1$ 
is an injective proper immersion, i.e., a proper embedding.   
\end{proof}

\begin{figure}[hbt] 
\begin{tabular}{cc}
\begin{minipage}{0.5\hsize} 
\centering
\includegraphics[width=0.7\hsize]{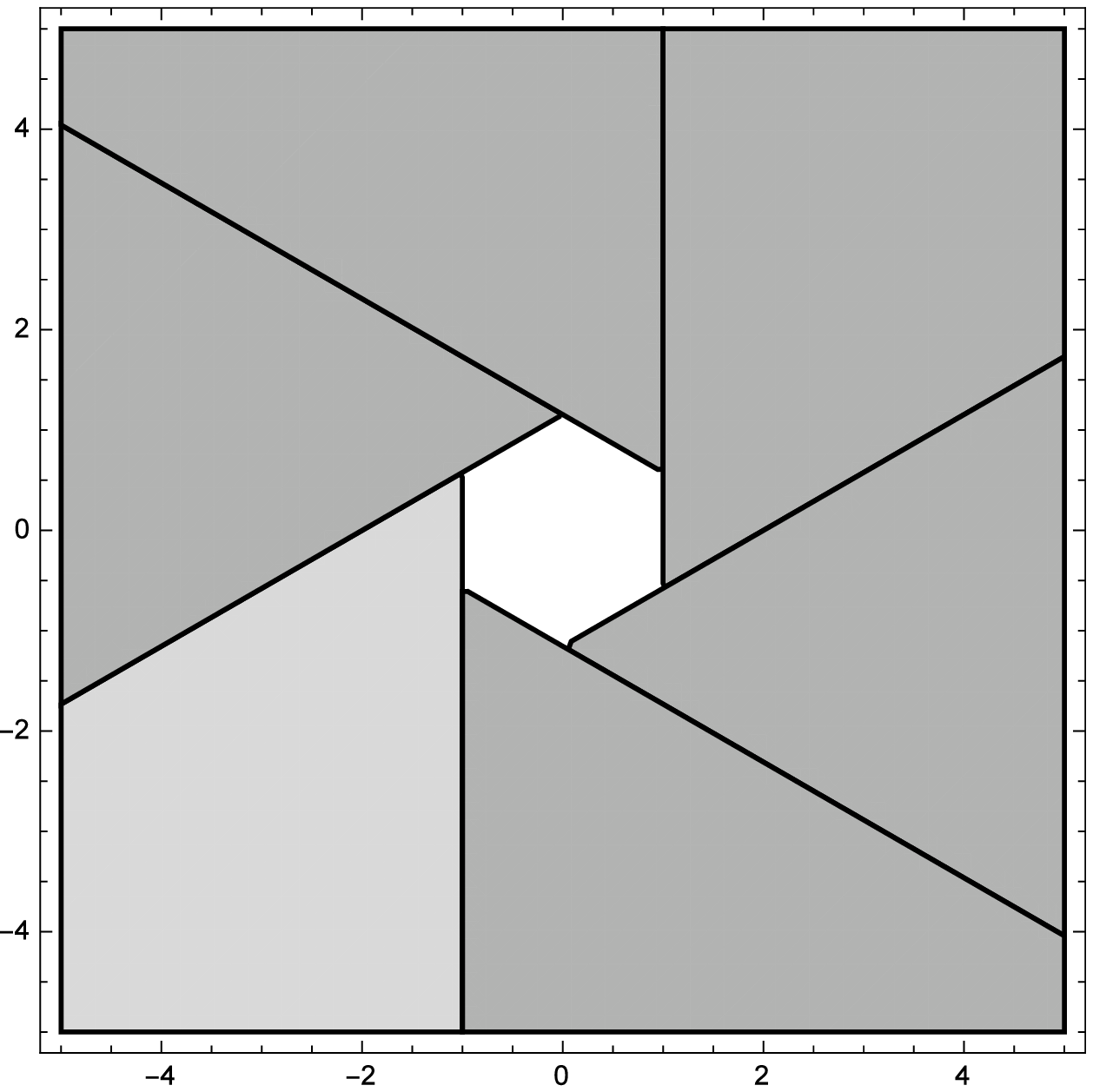}
\caption{$\bigcup_{k=0}^{n-1} R^k(D_h)$ ($n=6$, $h=1$)} 
\label{fig:region-6-h1}
\end{minipage} 
\begin{minipage}{0.5\hsize}
\centering 
\includegraphics[width=0.7\hsize]{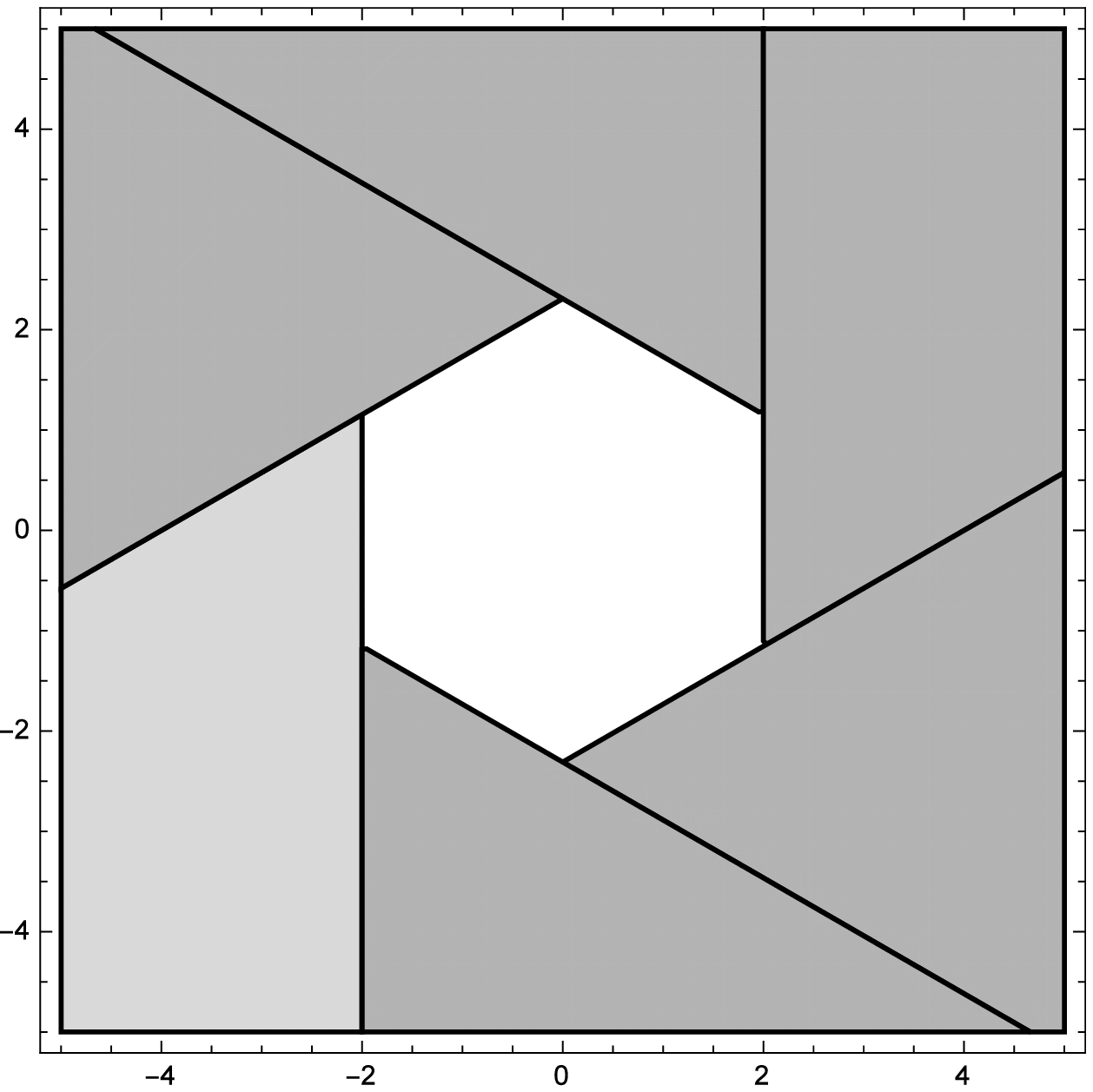}
\caption{$\bigcup_{k=0}^{n-1} R^k(D_h)$ ($n=6$, $h=2$)} 
\label{fig:region-6-h2}
\end{minipage} 
\end{tabular}
\end{figure}

\begin{figure}[htbp] 
\begin{center}
\raisebox{7mm}{\includegraphics[width=.28\linewidth]{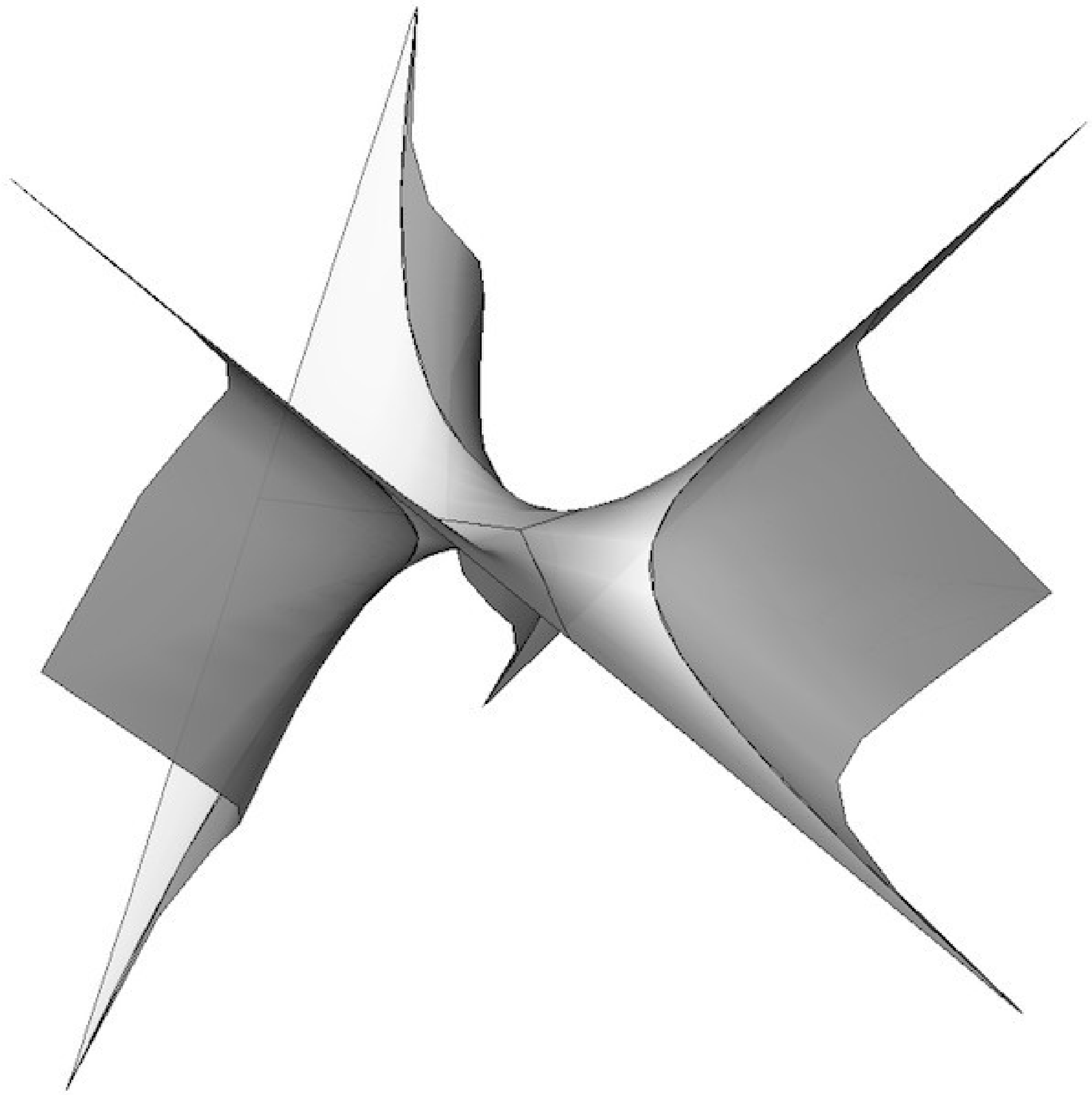}} 
\qquad
 \includegraphics[width=.34\linewidth]{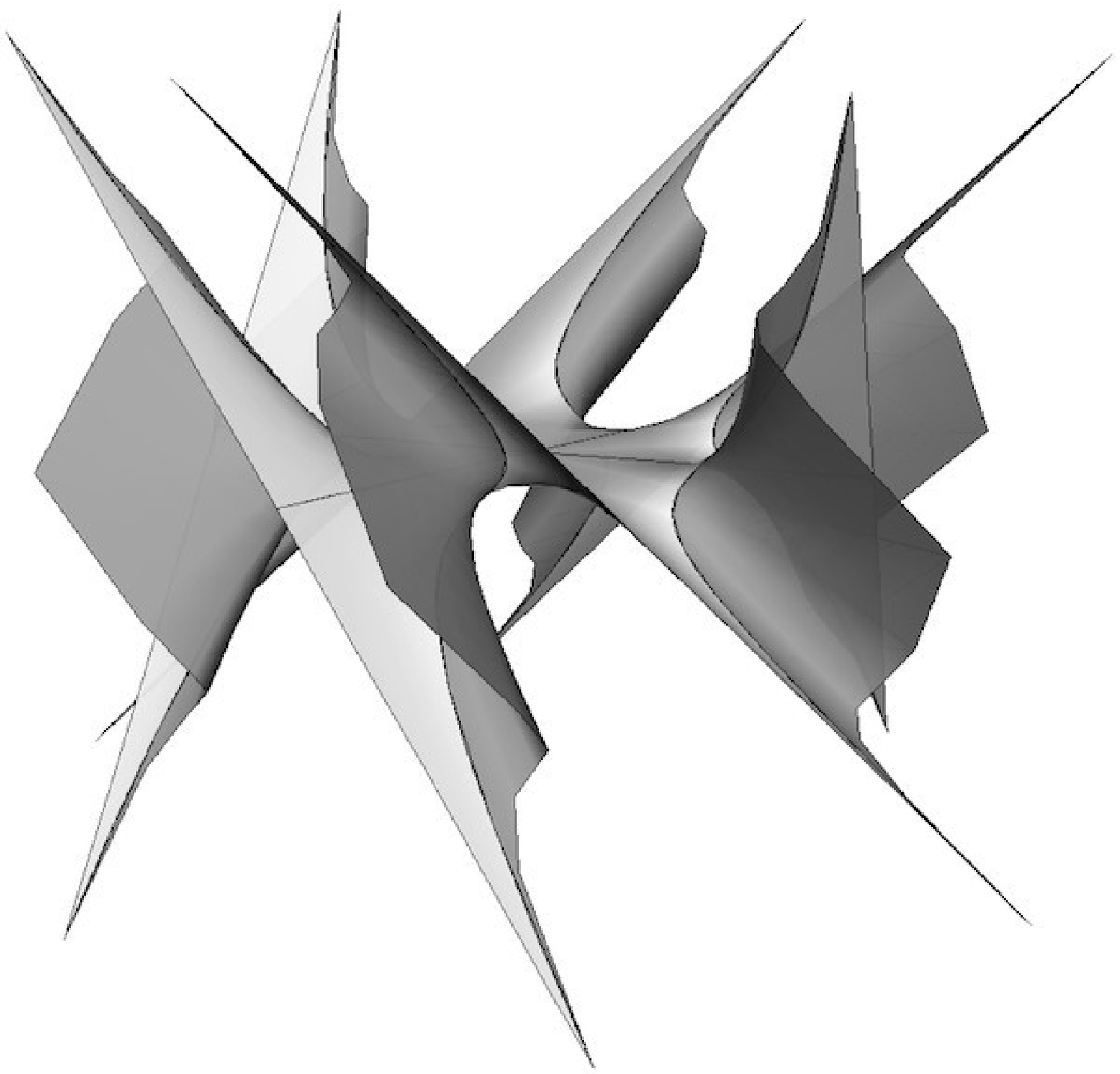} 
\caption{Images of $\tilde{f}_3$ and
$\tilde{f}_6$ (the time-like parts are
indicated by black shading). 
}
\label{fg:f5-psi}
\end{center}
\end{figure}

\appendix
\section{Some properties of Chebyshev polynomials}

The \emph{first Chebyshev polynomial} $T_n(x)$ ($n=1,2,\dots$) is, by definition, 
the polynomial of degree $n$ such that
\begin{equation}\label{eq:T}
    T_n(\cos \theta)=\cos n\theta.
\end{equation}
It holds that
\begin{equation}\label{eq:T2}
    T_n(u)=\frac{r^n+r^{-n}}{2}\qquad \left(u:=\frac{r+r^{-1}}2\right).
\end{equation}

\begin{lemma}\label{lemma:Phi}
 The following identity holds{\rm:}
\begin{equation}\label{eq:app:Phi}
\Psi_n(u):=    T_n(u)-\cos n\theta
       =2^{n-1}\prod_{j=0}^{n-1}
         \left(u-\cos\left(\theta-\frac{2\pi j}{n}\right)\right).
\end{equation}
\end{lemma}
\begin{proof}
By \eqref{eq:T}, we have
 \begin{align*}
   \Psi_n\left(\cos\left(\theta-\frac{2\pi j}{n}\right)\right)
    &= \cos\left(n\left(\theta-\frac{2\pi j}{n}\right)\right)-\cos n\theta
    =0.
 \end{align*}
 Since $\Psi_n(u)$ is a polynomial in $u$ of degree $n$ and
 the highest coefficient of $T_n(u)$
 is equal to $2^{n-1}$, we obtain the assertion.
\end{proof}

The \emph{second Chebyshev polynomial} $U_n(x)$ ($n=1,2,\dots$) is, by definition, 
the polynomial of degree $n$ such that
\begin{equation}\label{eq:U}
 \sin (n+1)\theta=U_{n}(\cos \theta)\sin \theta .
\end{equation}
It holds that
\begin{equation}\label{eq:U2}
    U_{n-1}(u) = \frac{r^{n}-r^{-n}}{r-r^{-1}}\qquad
       \left(u = \frac{r+r^{-1}}{2}\right).
\end{equation}

The first and the 
second Chebyshev polynomials are related as follows:
\begin{equation*}
 \frac{d}{d x} T_n(x) = n U_{n-1}(x). 
\end{equation*}

\begin{proposition}\label{prop:kaw}
For $m \ge 1$, it holds that
 \begin{equation}\label{eq:kaw}
  U_{2m}(x)-1 = 2 T_{m+1}(x) U_{m-1}(x). 
 \end{equation}
\end{proposition}

\begin{proof}
It is sufficient to show the identity for
$x=\cos \theta$ ($\theta\in [0,2\pi)$).
Then
\begin{align*}
 U_{2m}(\cos \theta)-1 &= \frac{\sin (2m+1) \theta}{\sin \theta} -1 
= \frac{\sin (2m+1) \theta - \sin \theta}{\sin \theta} \\
&= \frac{2 \cos(m+1) \theta \sin m \theta}{\sin \theta}
 = 2 T_{m+1}(\cos \theta) U_{m-1}(\cos \theta). 
\end{align*}
\end{proof}

\begin{proposition}\label{prop:app:u}
 Let $n$ be an integer greater than $2$ $($resp. $n=2)$. Then 
$y=U_{n-1}(x)$ is monotone increasing on the interval 
$\{ x \mid \cos \frac{\pi}{n-1} \le x < \infty \}$ and the range is 
$\{ y \mid -1 \le y < \infty \}$ 
$($resp. $\{ y \mid -2 \le y < \infty \})$. 
Furthermore, $U_{n-1}( \cos(\pi/n) )=0$ and $U_{n-1}( 1 )=n$
hold.  
\end{proposition}

\begin{corollary}\label{cor:app:u}
 For arbitrary $m \le n-1$, 
\begin{equation*}
 U_m(x) >0 \text{ for } \cos (\pi/n) < x < \infty. 
\end{equation*}
\end{corollary}

\begin{proposition}\label{prop:app:t}
 Let $n$ be an integer greater than or equal to $2$. Then 
$y=T_{n}(x)$ is monotone increasing on the interval 
$\{ x \mid \cos \frac{\pi}{n} \le x < \infty \}$ and the range is 
$\{ y \mid -1 \le y < \infty \}$. 
Furthermore, $T_{n}( \cos(\pi/2n) )=0$
and  $T_{n}( 1 )=1$ hold.  
\end{proposition}

 \begin{figure}[hbtp] 
\begin{tabular}{cc}
 \includegraphics[width=.45\linewidth]{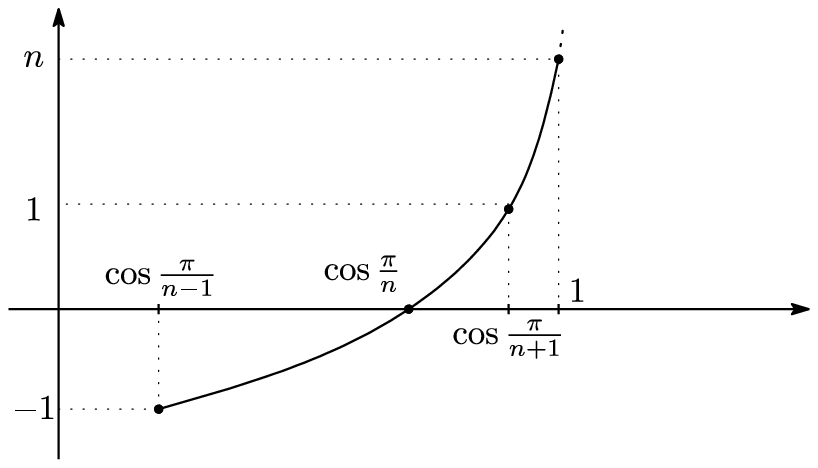} & 
 \includegraphics[width=.45\linewidth]{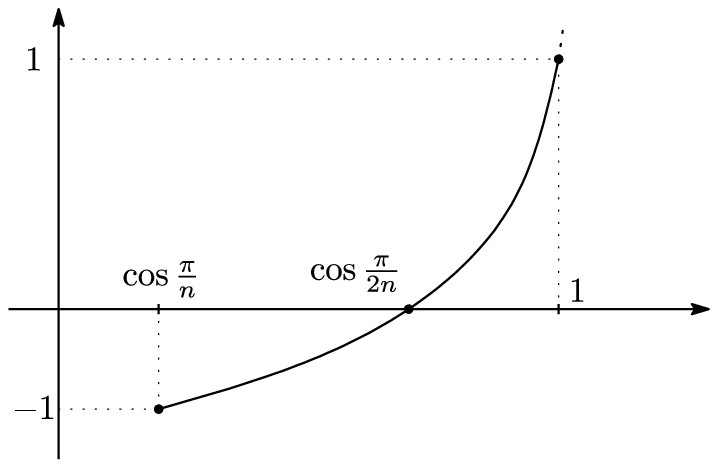} \\
 {$y=U_{n-1}(x)$}  & {$y=T_n(x)$} 
\end{tabular}
\caption{Chebyshev polynomials are monotone increasing 
on the interval toward the right. }
\end{figure} 



\begin{thebibliography}{20}
\bibitem{FKKRSUYY1}
S. Fujimori, Y. W. Kim, S.-E. Koh, W. Rossman, H. Shin,
M. Umehara, K. Yamada and  S.-D. Yang,
{\itshape
Zero mean curvature surfaces in Lorentz-Minkowski $3$-space which change type 
across a light-like line}, 
Osaka J. Math. {\bf 52} (2015), 285--297. 

\bibitem{FKKRSUYY2}
S. Fujimori, Y. W. Kim, S.-E. Koh, W. Rossman, H. Shin, M. Umehara,
K. Yamada and S.-D. Yang, 
{\it Zero mean curvature surfaces 
       in Lorentz-Minkowski $3$-space and 2-dimensional
       fluid mechanics}, 
Math. J. Okayama Univ. {\bf 57} (2015), 173--200.


\bibitem{FRUYY2}
S. Fujimori, W. Rossman, 
M. Umehara, K. Yamada and  S.-D. Yang,
{\it 
Embedded triply periodic zero mean curvature 
surfaces of mixed type in Lorentz-Minkowski $3$-space}, 
Michigan Math. J. {\bf 63} (2014), 189--207. 

\bibitem{FSUY}
  S. Fujimori, K. Saji, M. Umehara and K. Yamada,
  {\itshape Singularities of maximal surfaces},
  Math. Z. {\bf 259} (2008), 827--848.

\bibitem{G}
 C.~Gu, 
 {\itshape 
	The extremal surfaces in the $3$-dimensional
        Minkowski space},
  Acta Math. Sinica {\bf 1}
 (1985), 173--180.
\bibitem{JM}
 L. P. Jorge and W. H.  Meeks, III,
 {\itshape The topology of complete minimal surfaces of finite 
           total Gaussian curvature},
 Topology {\bf 22} (1983), 203--221.

\bibitem{KKSY}
 Y. W. Kim, S.-E Koh, H. Shin and S.-D. Yang,
  {\itshape Spacelike maximal surfaces, 
timelike minimal surfaces, and Bj\"{o}rling 
  representation formulae}, 
  J. Korean Math. Soc. {\bf 48} (2011), 1083--1100. 
\bibitem{Kl}
  V. A. Klyachin, 
  {\itshape Zero mean curvature surfaces of mixed type in Minkowski space}, 
  Izvestiya Math. {\bf 67} (2003), 209--224.

\bibitem{K1}
 O. Kobayashi,
    {\itshape Maximal surfaces in the $3$-dimensional
        Minkowski space $\mathbb{L}^3$},
    Tokyo J. Math. {\bfseries 6} (1983), 297--309.
\bibitem{ST}
 V. Sergienko and V. G. Tkachev, 
{\it Doubly periodic maximal surfaces with singularities}, 
Proceedings on Analysis and Geometry (Russian) (Novosibirsk Akademgorodok,
1999), 571--584, Izdat. Ross. Akad. Nauk Sib. Otd. Inst. Mat., 
Novosibirsk, 2000.
%
\bibitem{UY}
 M. Umehara and K. Yamada,
 {\itshape
     Maximal surfaces with singularities in
     Minkowski space},
      Hokkaido Math. J. {\bfseries 35} (2006), 13--40.
\end{thebibliography}
\end{document}